\documentclass[pdflatex,sn-mathphys-num]{sn-jnl}

\usepackage{graphicx}%
\usepackage{multirow}%
\usepackage{amsmath,amssymb,amsfonts}%
\usepackage{amsthm}%
\usepackage{mathrsfs}%
\usepackage[title]{appendix}%
\usepackage{xcolor}%
\usepackage{textcomp}%
\usepackage{manyfoot}%
\usepackage{booktabs}%
\usepackage{algorithm}%
\usepackage{algorithmicx}%
\usepackage{algpseudocode}%
\usepackage{listings}%

\usepackage{stmaryrd}
\usepackage{epsfig}
\usepackage{epstopdf}
\usepackage{tikz}
\usepackage{pgfplots}
\usepackage{caption}
\usepackage[normalem]{ulem}

\theoremstyle{thmstyleone}%
\newtheorem{theorem}{Theorem}
\theoremstyle{thmstyletwo}%
\newtheorem{remark}{Remark}%

\theoremstyle{thmstylethree}%
\newtheorem{definition}{Definition}%

\newcommand{\p}{\partial}
\newcommand{\bn}{{\bf n}}
\newcommand{\bPhi}{{\bf \Phi}}
\newcommand{\bPi}{{\bf \Pi}}
\newcommand{\bPsi}{{\bf \Psi}}

\newcommand{\cA}{{\mathcal A}}
\newcommand{\cB}{{\mathcal B}}

\newcommand{\cK}{{\mathcal K}}

\newcommand{\cN}{{\mathcal N}}
\newcommand{\cP}{{\mathcal P}}
\newcommand{\cR}{{\mathcal R}}

\newcommand{\cS}{{\mathcal S}}
\newcommand{\cT}{{\mathcal T}}
\newcommand{\cV}{{\mathcal V}}
\newcommand{\cW}{{\mathcal W}}

\newcommand{\lifting}{\mathbf{ L}}

\newcommand{\spaceFEM}{V_{hp}(\Omega^1)}
\newcommand{\spaceFEMgradstrip}{{\textbf{W}_{hp}(\Omega^1)}}
\newcommand{\spaceFEMgradstripcomp}{{{W}_{hp}(\Omega^1)}}
\newcommand{\spaceFEMgrad}{\textbf{W}_{hp}(\Omega^1)}
\newcommand{\spaceFEMgradK}{\textbf{W}_{hp}(K)}

\newcommand{\spaceFEMsig}{V^{\sigma,\mu}_{hp}(\Omega^1)}
\newcommand{\spaceFEMgradsig}{\textbf{W}^{\sigma,\mu}_{hp}(\Omega^1)}

\newcommand{\spaceBEM}{V_{hp}(\Gamma^2)}
\newcommand{\spaceBEMtrac}{W_{hp}(\Gamma^2)}
\newcommand{\spaceFBEM}{V_{hp}}
\newcommand{\spaceFBEMsig}{V^{\sigma,\mu}_{hp}}

\newcommand{\spaceBEMsig}{V^{\sigma,\mu}_{hp}(\Gamma^2)}
\newcommand{\spaceBEMgradsig}{W^{\sigma,\mu}_{hp}(\Gamma^2)}

\newcommand{\fspace}{V}
\newcommand{\fspaceFEM}{V(\Omega^1)}
\newcommand{\fspaceBEM}{V(\Gamma^2)}

\newcommand{\Gstrip}{{\Omega^1}}

\newcommand{\resphi}{R_{hp}}
\newcommand{\res}[1]{{R_{hp,{#1}}}}

\newcommand{\invrate}{\mathcal G}

\newcommand{\meshFEM}{{\mathcal T_{hp}(\Omega^1)}}

\newcommand{\meshFEMsig}{{\mathcal T_{hp}^{\sigma}(\Omega^1)}}
\newcommand{\meshFEMsigO}{{\hat{\mathcal{T}}_{hp}^{\sigma}(\Omega^1)}}

\newcommand{\meshBEM}{{\mathcal T_{hp}(\Gamma^2)}}
\newcommand{\meshBEMsig}{{\mathcal T^{\sigma}_{hp}(\Gamma^2)}}

\newcommand{\meshIone}{{\mathcal T_{hp}(\Gamma^1_I)}}
\newcommand{\meshItwo}{{\mathcal T_{hp}(\Gamma^2_I)}}

\newcommand{\bilt}{A}
\newcommand{\bilinc}{a_{hp}}
\newcommand{\bilint}{\bilt_{hp}}

\newcommand{\normp}[1]{{1/2,{#1}}}
\newcommand{\normm}[1]{{-1/2,{#1}}}
\newcommand{\norme}[1]{{{hp},{#1}}}

\newcommand{\bx}{{\bf x}}
\newcommand{\by}{{\bf y}}

\newcommand{\frA}{{\mathfrak A}}
\newcommand{\frB}{{\mathfrak B}}
\newcommand{\frC}{{\mathfrak C}}

\newcommand{\frR}{{\mathfrak R}}
\newcommand{\frS}{{\mathfrak S}}

\newcommand{\R}{{\mathbb R}}

\newtheorem{corollary}{Corollary}[section]
\newtheorem{lemma}{Lemma}[section]

\allowdisplaybreaks[3]

\def\NN{\mathbb{N}}
\def\RR{\mathbb{R}}

\def \bea  {\begin{eqnarray}}
\def \eea  {\end{eqnarray}}
\def \bean {\begin{eqnarray*}}
\def \eean {\end{eqnarray*}}

\def \la {\langle}
\def \ra {\rangle}

\DeclareMathOperator{\supp}{supp}

\DeclareMathOperator{\divv}{div}
\newcommand{\jump}[1]{[#1]}

\newcommand{\mbf}[1]{\mathbf{#1}}
\newcommand{\ud}{\,\mathrm{d}}

\newcommand{\ltwop}[2]{\left( #1  , #2 \right)}

\newcommand{\dualp}[2]{\la #1 ,#2 \ra}
\newcommand{\eqspl}[2]{\begin{equation}\label{#1}\begin{split} #2 \end{split}\end{equation}}

\newcommand{\alil}[2]{\begin{align}\label{#1} #2 \end{align}}
\newcommand{\norm}[2]{\left\| #1 \right\|_{#2}}

\title{Nonconforming $hp$-FE/BE coupling on unstructured meshes based on Nitsche's method}

\author*[1]{\fnm{Alexey} \sur{Chernov}}\email{alexey.chernov@uni-oldenburg.de}
\equalcont{These authors contributed equally to this work.}

\author*[2]{\fnm{Peter} \sur{Hansbo}}\email{peter.hansbo@ju.se}
\equalcont{These authors contributed equally to this work.}

\author*[3]{\fnm{Erik Marc} \sur{Schetzke}}\email{erik.marc.schetzke@fkie.fraunhofer.de}
\equalcont{These authors contributed equally to this work.}

\affil*[1]{\orgdiv{Institut f\"ur Mathematik}, \orgname{Carl von Ossietzky Universit\"at}, \orgaddress{\street{Ammerl\"ander Heerstra\ss{}e 114-118}, \city{Oldenburg}, \postcode{26129}, 
\country{Germany}}}

\affil*[2]{\orgdiv{Department of Mechanical Engineering}, \orgname{J\"onk\"oping University}, \orgaddress{
 \city{J\"onk\"oping}, \postcode{SE-551 11}, \country{Sweden}}}

\affil*[3]{\orgdiv{Communication systems}, \orgname{Fraunhofer FKIE}, \orgaddress{\street{Zanderstra\ss e 5}, \city{Bonn}, \postcode{53177}, 
\country{Germany}}}

\abstract{
We construct and analyse a $hp$-FE/BE coupling on non-matching meshes, 
based on Nitsche's method. Both the mesh size and the polynomial degree are 
changed to improve accuracy. Nitsche's method leads to a positive definite
formulation, thus, unlike the mortar method, it does not require the 
\textit{Babu\v{s}ka-Brezzi} condition for stability. The method is stable
provided the stabilization function is larger than a certain threshold.
We obtain an \textit{explicit} bound for the threshold and derive 
\textit{a priori} error estimates. Our analysis can be easily extended to the
pure FE or the pure BE decomposition as well as to the case of more than two
subdomains. The problem in a bounded domain is considered in detail, but the
case of an unbounded BE subdomain and a bounded FE subdomain follows with
similar arguments. We develop convergence analysis and provide numerical examples for 
quasi-uniform as well as geometrically refined \emph{hp} discretisations in both subdomains
with analytic and singular solutions. 
}

\begin{document}
\maketitle

\section{Introduction}
The focus of the present work is the analysis of Nitsche-type coupling in the $hp$-setting on
non-matching meshes. In particular, we derive stability conditions and a priori error estimates
with explicit dependence on both the local mesh size and polynomial degree, including the case
of geometrically graded $hp$ discretisations for singular solutions. The resulting method is
positive definite and avoids the inf--sup stability constraints that arise in mortar
formulations.

Coupling finite element and boundary element methods is a classical approach for the numerical
treatment of interface problems and problems posed on unbounded domains. Early formulations
were based on conforming discretisations or on saddle point formulations employing Lagrange
multipliers, leading to mortar-type methods that require the verification of inf--sup stability
conditions, cf., e.g., Bernardi, Maday, and Patera \cite{bernardi1993domain}. While such
approaches are well understood and widely used, their stability analysis and practical
implementation become increasingly delicate in the presence of non-matching meshes and variable
approximation orders.

An alternative to mortar formulations is provided by Nitsche's method \cite{Nit70}, originally
introduced for the weak imposition of boundary conditions in finite element methods and later
extended to domain decomposition and interface problems. In the context of FE/BE coupling,
Nitsche-type formulations offer the advantage of yielding positive definite systems while
avoiding additional multiplier spaces. This line of research has been pursued by several
authors, including Becker, Hansbo, and Stenberg \cite{BeHaSg03}, as well as Heinrich and
coworkers \cite{HeNi03,HeJu06}, among others. These works establish stability and convergence
for fixed-order discretisations and provide a solid foundation for weakly enforced coupling
conditions.

At the same time, the development of $hp$-finite element methods has demonstrated that
combining mesh refinement with variable polynomial degree can lead to exponential convergence
rates, particularly for problems with singular solutions. The analysis of $hp$-methods,
however, relies on refined inverse estimates and weighted Sobolev space techniques, and
extending coupling strategies such as Nitsche's method to the $hp$-setting is therefore 
non-trivial. In particular, stability conditions that are adequate for fixed-order
discretisations do not automatically carry over to non-uniform $hp$-meshes or geometrically
graded refinements.

The present work addresses this gap by analysing a Nitsche-type FE/BE coupling in an 
$hp$-framework on non-matching meshes. The emphasis is not on proposing a new coupling
paradigm, but on establishing stability conditions and a priori error estimates that are
explicit in both the local mesh size and the polynomial degree. Special attention is paid to
the choice of the stabilisation parameter, which is shown to depend locally on inverse
inequalities rather than on global quasi-uniformity assumptions. This makes the approach
compatible with standard $hp$-strategies for singularly perturbed or corner-dominated problems.

Related Nitsche-based formulations for FE/BE coupling have also been developed in the work of
Betcke and co-authors \cite{betcke2022,betcke2019boundary}. While sharing the use of weakly
imposed coupling conditions, these approaches address complementary questions and are primarily
concerned with flexibility of discretisations and algorithmic robustness. In contrast, the
present analysis focuses on the approximation and stability properties of Nitsche-type coupling
in the $hp$-setting, and in particular on the explicit dependence of stability conditions and
error estimates on local mesh size and polynomial degree.

The remainder of the paper is organised as follows. We first introduce the model problem,
notation, and functional analytic setting underlying the proposed formulation. We then describe
the $hp$--Nitsche coupling and its discretisation on non-matching meshes. Stability of the
method and appropriate choices of the stabilisation parameter are analysed next, followed by a
priori error estimates. Finally, numerical experiments are presented to illustrate the
theoretical results.

\subsection{Problem setting and analytical framework}

In this section we introduce the model problem, notation, and functional analytic setting used throughout the paper. The material is largely standard in the context of FE/BE coupling and is collected here to fix notation.

Let $\Omega$ be a bounded open polygonal domain $\Omega\subset\RR^2$,
decomposed into two simply connected mutually disjoint domains, such that 
\[
	\overline{\Omega} = \overline{\Omega^1} \cup \overline{\Omega^2}.
\]
Let $\Gamma := \partial\Omega, \Gamma^i := \partial\Omega^i,i=1,2,
\Gamma_I := \Gamma^1\cap\Gamma^2$ and $\bn$ be the unit outward normal to 
$\Gamma$. As a model problem we take the diffusion equation in $\Omega$ with
mixed boundary conditions: find $u: \overline{\Omega}\rightarrow\RR$ solving
\begin{equation}\label{p1}
\begin{split}
		-\divv(\kappa\nabla u) &= f\quad \mbox{in}\; \Omega, \\
		u&= 0\quad \mbox{on}\; \Gamma_D,\\
		\kappa\nabla u\cdot\bn &= g \quad \mbox{on}\; \Gamma_N,
\end{split}
\end{equation}
with a disjoint decomposition 
$\overline{\Gamma} = \overline{\Gamma}_D \cup \overline{\Gamma}_N$, 
prescribed volume force $f$ and normal derivative $g$. If the solution $u$ is
sufficiently smooth along $\Gamma_I$, problem \eqref{p1} in $\Omega$
is equivalent to the following interface problem, cf. \cite{Babu70}:\\

\noindent
\begin{minipage}{0.42\textwidth}
Find $u:\overline{\Omega}\rightarrow\RR$, such that\\
\begin{equation}\label{p2}
\begin{split}
		-\divv(\kappa\nabla u) &= f\quad \mbox{in}\;\Omega,\\
		u&=0\quad \mbox{on}\; \Gamma_D,\\
		\kappa\nabla u\cdot\bn &= g\quad \mbox{on}\; \Gamma_N,\\
		\jump{u}=0,\quad \jump{\kappa\nabla u\cdot\bn}&=0
		\quad\mbox{on}\; \Gamma_I. 
\end{split}		
\end{equation}
\end{minipage}
\hspace{0.05\textwidth}
\begin{minipage}{0.5\textwidth}
\begin{tikzpicture}[scale=1.8]
\definecolor{ilpablue}{rgb}{0.16,0.43,0.75}
\definecolor{dgreen}{rgb}{0.0,0.39,0.0}
\draw[ilpablue, ultra thick, -] (1,-0.5) -- (1.3,0.7) -- (-0.3, 1.2) -- (-1.8, 0.8) -- (-0.87, -0.42); 
\draw[black, ultra thick, -] (-0.87, -0.42) -- (0,0) -- (1,-0.5); 
\draw[dgreen, dashed, ultra thick] (0,0) -- +(-0.3, 1.2); 
\draw (0.5,0.5) node[text=black] {$\Omega^2$};
\draw (-0.85,0.3) node[text=black] {$\Omega^1$};
\draw (0.06,0.5) node[text=black!60!green] {$\Gamma_I$};
\draw (-0.03,-0.2) node[text=black] {$\Gamma_D$};
\draw (-0.35,1.4) node[text=ilpablue] {$\Gamma_N$};
\draw[black,->] (1.14,0.75) -- (1.20,0.99); 
\draw  (1.3,0.89) node[text=black] {$\textbf{n}^2$};
\draw[black,->] (-1.18, -0.013) -- (-1.354, -0.097); 
\draw  (-1.23,-0.163) node[text=black] {$\textbf{n}^1$};
\draw (-2.0, 0.8) node[text=black] {$A_j$};
\draw (1.2,-0.5) node[text=black] {$A_i$};
\end{tikzpicture}

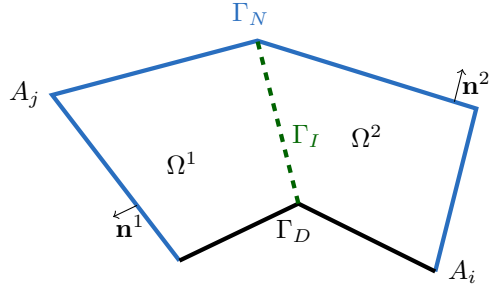
\captionof{figure}{Domain decomposition of $\Omega = \Omega^1 \cup \Omega^2$ with 
$\overline{\Gamma} = \overline{\Gamma_D} \cup \overline{\Gamma_N},$ and 
$\Gamma_I = \partial \Omega^1 \cap \partial \Omega^2$.}
\label{fig:domdecomp}
\end{minipage}
$~$\\

The solution of \eqref{p2} is considered as a pair $u = (u^1,u^2)$, where 
$u^i := u|_{\Omega^i}$. Then the jumps across the interface $\Gamma_I$ are
defined by
\[
	\jump{u} := u^1|_{\Gamma_I} - u^2|_{\Gamma_I},\qquad 
	\jump{\kappa\nabla u\cdot \bn} := \nabla u^1|_{\Gamma_I}\cdot\bn^1 
	+ \nabla u^2|_{\Gamma_I}\cdot\bn^2,
\]
where $\bn^i,i=1,2$ denotes the outward normal vector field to $\Gamma^i$. 
For brevity for $i=1,2$ and $A \in \{D,N\}$ we denote 
$\Gamma^i_A := \Gamma_A\cap\Gamma^i$ and $\Gamma^2_F:= \Gamma^2_N\cup\Gamma_I$.
For some $\omega\subset\RR^r,r=1,2$ and its boundary (part) 
$\gamma\subseteq\partial\omega$ and sufficiently smooth functions
$\mbf{u},\mbf{v}: \omega\rightarrow\RR,\; u,v:\omega\rightarrow\RR$ we define
\[
	\ltwop{\mbf{u}}{\mbf{v}}_\omega := \int_\omega \mbf{u}\cdot\mbf{v}\ud x,
	\quad
	\ltwop{u}{v}_\omega := \int_\omega uv\ud x,
	\quad
	\dualp{u}{v}_\gamma := \int_\gamma uv\ud s,
\]
which can be regarded as scalar product and/or duality pairing on the
corresponding Sobolev spaces on $\omega$ and $\gamma$. 

We  shall consider finite elements on $\Omega_1$ and boundary elements on $\Gamma_2$ below, thus defining discrete subspaces to
$V = V(\Omega^1) \times V(\Gamma^2)$ with 
\begin{equation*}
\begin{split}
	V(\Omega^1) &:= \left\{ u\in H^1(\Omega^1): u|_{\Gamma^1_D} = 0\right\},\\
	V(\Gamma^2) &:= \left\{ u\in H^{1/2}(\Gamma^2): 
	\supp u\subset\Gamma^2_F\right\} = \tilde{H}^{1/2}(\Gamma^2_F).
\end{split}
\end{equation*}
For the diffusion coefficient $\kappa$ we assume
\begin{equation}\label{kappa_bnd}
\begin{split}
	\exists\kappa_{\min},\kappa_{\max}\in(0,\infty): 
	\kappa_{\min}\leq \kappa(\mbf{x})\leq\kappa_{\max}\quad&
	\mbox{a.e. in } \Omega^1,\\
	\kappa(\mbf{x}) = 1\quad&\mbox{a.e. in } \Omega^2.
\end{split}
\end{equation}

For simplicity we will assume that $\kappa \in L^\infty(\Omega^1)$ is piecewise
constant on the underlying mesh, i.e.
\[
\kappa|_K = \kappa_K, \qquad \forall \text{ elements } K.
\]
Then we have in $\Omega^1$
\begin{equation*}
	\kappa_{\min}\norm{\nabla u}{L^2(\Omega^1)}^2 \leq 
	\ltwop{\kappa\nabla u}{\nabla u}_{\Omega^1} \leq
	\kappa_{\max}\norm{\nabla u}{L^2(\Omega^1)}^2
	\qquad\forall u\in H^1(\Omega^1),
\end{equation*}
and the subproblem in $\Omega^2$ admits boundary reduction via Green's formula
and \eqref{eq:intCalderon} below
\begin{equation}\label{Green}
	\ltwop{\nabla u^2}{\nabla \varphi^2}_{\Omega^2} 
	- \ltwop{f}{\varphi^2}_{\Omega^2} 
	= \dualp{\cS u^2 -\cN f}{\varphi^2}_{\Gamma^2}.
\end{equation}
Here and in what follows  we abbreviate the Dirichlet trace $u^2|_{\Gamma^2}$ 
by $u^2$, when it is clear from the context. In \eqref{Green}, $\cS$ is the
Steklov-Poincar\'e operator and $\cN$ is the Newton potential, defined as
follows. Let
\[
	G(\mbf{x},\mbf{y}) = -\frac{1}{2\pi}\log \left| \mbf{x}-\mbf{y} \right|
\]
be the fundamental solution of the Laplace equation in $\RR^2$. Then
\begin{equation}\label{NSdef}
	\cN := (\cK^\prime+1/2)\cV^{-1} \cN_0 - \cN_1,
	\qquad \cS := \cW + (\cK^\prime+1/2)\cV^{-1}(\cK+1/2),
\end{equation}
where for $\mbf{x}\in\Gamma^2$ we set
\begin{align*}
	(\cN_0 f)(\mbf{x}) &:= 
	\int_{\Omega^2} G(\mbf{x},\mbf{y}) f(\mbf{y}) \ud \mbf{y}, 
	&(\cN_1 f)(\mbf{x}) &:= 
	\int_{\Omega^2} (\nabla_{\mathbf{x}} G(\mbf{x},\mbf{y}) \cdot	
	\bn^2(\mathbf{x}) f(\mbf{y}) \ud \mbf{y},\\
	(\cV u)(\mbf{x}) &:= 
	\int_{\Gamma^2} G(\mbf{x},\mbf{y}) u(\mbf{y}) \ud s_{\mbf{y}},
	&(\cK u)(\mbf{x}) &:= 
	\int_{\Gamma^2} (\nabla_{\mbf{y}}G(\mbf{x},\mbf{y})\cdot \bn^2(\mbf{y})) 	
	u(\mbf{y}) \ud s_{\mbf{y}},\\
	(\cW u)(\mbf{x}) &:= 
	-\nabla_{\mbf{x}}(\cK u)(\mbf{x})\cdot \bn^2(\mbf{x}),
	&(\cK^\prime u)(\mbf{x}) &:= 
	\int_{\Gamma^2} (\nabla_{\mbf{x}}G(\mbf{x},\mbf{y})\cdot \bn^2(\mbf{x}))
	u(\mbf{y}) \ud s_{\mbf{y}}.
\end{align*}	
For later reference and as a means to derive the symmetric representation of the
Steklov-Poincar\'e operator, we recall the Calder\`on system for the trace and
conormal derivative on the boundary element domain $\Gamma^2$. For the interior
Poisson problem there holds, cf. e.g. \cite[p. 137]{Stb08},
\begin{equation}\label{eq:intCalderon}
	\begin{pmatrix}
		u^2\\ \frac{\p u^2}{\p \bn^2} 
	\end{pmatrix}
	= 
	\begin{pmatrix}
	\tfrac{1}{2} - \mathcal{K}& \mathcal{V}\\
	\mathcal{W}& \frac{1}{2} + \mathcal{K}'
	\end{pmatrix}
	\begin{pmatrix}
		u^2\\ \frac{\p u^2}{\p \bn^2} 
	\end{pmatrix}
	+
	\begin{pmatrix}
		\mathcal{N}_0 f\\ 		\mathcal{N}_1 f
	\end{pmatrix}.
\end{equation}

We refer to \cite{CarSt95,Csb88,StrSwb04,St04} and references therein for
further details. The properties of the above boundary integral operators yield
continuity and coercivity of $\cS$ on $V(\Gamma^2)$:
\begin{equation}\label{coel_S}
\begin{array}{lclr}
	\exists c_\cS>0:\dualp{\cS u}{u}_{\Gamma^2} &\geq& 
	c_\cS \norm{u}{H^{1/2}(\Gamma^2)}^2, &\forall u\in V(\Gamma^2),\\[1ex]
	\exists C_\cS>0:\dualp{\cS u}{v}_{\Gamma^2} &\leq& 
	C_\cS \norm{u}{H^{1/2}(\Gamma^2)} \norm{v}{H^{1/2}(\Gamma^2)},
	&\forall u,v\in V(\Gamma^2).
\end{array}
\end{equation}

\subsection{Discretisation}\label{subsec:Discr}
We use a finite element discretisation in $\Omega^1$ and a boundary element discretisation on $\Gamma^2$. Note however, that construction and
convergence in case of finite element discretisations in both subdomains $\Omega^1$ and $\Omega^2$ follow with the same arguments. 
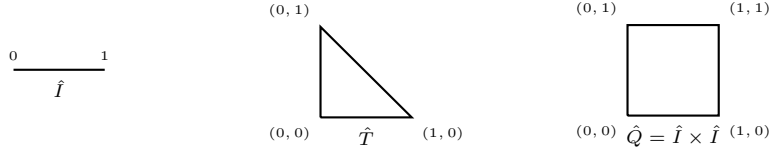
\begin{figure}[t]
\centering
\begin{minipage}{0.3\textwidth}
\begin{center}
\begin{tikzpicture}[scale=1.2]
\draw[black, thick,-] (0,0) node[above] {\tiny{$0$}}-- (0.5,0) node[below] {\footnotesize{$\hat I$}} -- (1,0) node[above] {\tiny{$1$}};
\end{tikzpicture}
\end{center}
\end{minipage}
\begin{minipage}{0.3\textwidth}
\begin{center}
\begin{tikzpicture}[scale=1.2]
\draw[black, thick, -] (0,0) node[below left] {\tiny{$\left( 0,0\right)$}}-- (0,1) node[above left] {\tiny{$\left( 0,1\right)$}} -- (1,0) node[below right]
{\tiny{$\left( 1,0\right)$}}
-- (0.5,0) node[below] {\footnotesize{$\hat T$}} -- (0,0);
\end{tikzpicture}
\end{center}
\end{minipage}
\begin{minipage}{0.3\textwidth}
\begin{center}
\begin{tikzpicture}[scale=1.2]
\draw[black, thick, -] (0,0) node[below left] {\tiny{$\left(0,0\right)$}} -- (0,1) node[above left] {\tiny{$\left(0,1\right)$}}
-- (1,1) node[above right] {\tiny{$\left(1,1\right)$}} -- (1,0) node[below right] {\tiny{$\left(1,0\right)$}} 
-- (0.5,0) node[below] {\footnotesize{$\hat Q = \hat I \times \hat I$}}
 -- (0,0);
\end{tikzpicture}\\
\end{center}
\end{minipage}
\caption{Reference finite and boundary elements in 1 and 2 dimensions} \label{fig:elem123}
\end{figure}
We call $\hat{K}$ a reference finite (resp. boundary) element, if $\hat{K}$ is
a (possibly trivial) Cartesian product of $m=1$ or $2$ reference simplices, cf. Figure
\ref{fig:elem123}.
In the following we shall consider affine meshes $\mathcal T_{hp}$ in 
$\Omega^1$ and $\Gamma^2$, respectively. Define
\begin{equation}
	\begin{split}
		 \mathcal T_{hp}(\Omega^1) 
		   &:= \left\{ K = F_K(\hat{K})\subset\Omega^1, F_K \mbox{ affine }, 
		   \hat{K}=\hat{T} \mbox{ or } \hat{Q} \right\},\\
		 \mathcal T_{hp}(\Gamma^2) 
		   &:= \left\{ K = F_K(\hat{I})\subset\Gamma^2, F_K \mbox{ affine } \right\}.
	\end{split}
\end{equation}

Set $h_K = \text{diam}(K)$ and assume that $\cT_{hp}(\Omega^1)$ is shape regular, i.e. there
exists a constant $\tau>0$, such that
\[
	\frac{h_K}{\varrho_K}\leq \tau \qquad\forall K\in\cT_{hp}(\Omega^1),
\]
where $\varrho_K = \sup\{\text{diam}(B): B \text{ is a ball contained in } K \}$. Moreover, we assume that $\cT_{hp}(\Omega^1)$ is regular.
Thus, for every $K,K'\in\cT_{hp}(\Omega^1)$ the intersection $K\cap K'$ is either empty, consists of a common vertex or an entire edge of $K$.

Let $\cP_p(T)$ be the space of polynomials on a triangle $T$ of total degree 
$\leq p$.
We define a (product) space of polynomials on the respective reference element $\hat{K}$ via
\[
	\cR_{p_K}(\hat{K}) := \bigotimes^m_{i=1}\cP_{p_K}(T_i) = 
	\begin{cases}
		\cP_{p_K}(\hat{I}), & \hat{K}=\hat{I},\\
		\cP_{p_K}(\hat{T}), & \hat{K}=\hat{T},\\
		\cP_{p_K}(\hat{I})\otimes\cP_{p_K}(\hat{I}), & \hat{K} = \hat{Q}.
	\end{cases}
\]

Based on the definitions above, we introduce a conforming discretisation of $V$ by continuous
piecewise polynomials $V_{hp} := V_{hp}(\Omega^1) \times V_{hp}(\Gamma^2) \subset V$
\begin{equation}
\begin{split}\label{def:hpspaces}
	V_{hp}(\Omega^1) &:= \left\{ \Phi \in V(\Omega^1): 
	\Phi\circ F_K \in \mathcal{R}_{p_K}(F^{-1}_K(K)),\;\forall K\in \meshFEM\right\},\\
	V_{hp}(\Gamma^2) &:= \left\{ \Phi \in V(\Gamma^2): 
	\Phi \circ F_K \in \mathcal{P}_{p_K}(F^{-1}_K(K)),\;\forall K\in \meshBEM\right\}.\\
\end{split}
\end{equation}
Note that the homogenous Dirichlet boundary conditions are incorporated into $V_{hp}(\Omega^1)$ 
and $V_{hp}(\Gamma^2)$. For gradients in $V_{hp}(\Omega^1)$ we define conforming
discretisations $\mbf{W}_{hp}(\Omega^1)\subset \left[L^2(\Omega^1)\right]^2$ 
and $W_{hp}(\Gamma^2)\subset H^{-1/2}(\Gamma^2)$ for 
normal derivatives on $\Gamma^2$ by
\begin{equation}\label{def:gradientspaces}
\begin{split}
	\mbf{W}_{hp}(\Omega^1) &:= \left\{ \Phi\in\left[ L^2(\Omega^1)\right]^2: 
	\Phi\circ F_K\in\left[\cR_{p_K}(F^{-1}_K(K))\right]^2,\,
	\forall K\in \meshFEM \right\},\\[2mm]
	W_{hp}(\Gamma^2) &:= \left\{ \Phi\in L^2(\Gamma^2): \Phi \circ F_K \in 
	\cP_{p_K}(F^{-1}_K(K)),\, \forall K\in \meshBEM\right\}.
\end{split}
\end{equation}
For the analysis of the $hp$-Nitsche method for problems with corner singularities and
geometric mesh refinement we briefly outline the discretisation. 
Let the corners of $\Omega$ be enumerated as $A_1,\dots,A_M\in\mathbb{R}^2$,
such that $\overline\Gamma = \cup^M_{j=1} \overline\Gamma_j$
with straight arcs $\Gamma_j$ as $\Omega$ is a polygonal domain,
(cf. Figure \ref{fig:domdecomp}) with possible singularities situated at $A_i$.
We introduce geometric meshes on $\Omega^1$ and $\Gamma^2$, namely $\meshFEMsig$ and 
$\meshBEMsig$ for given $\sigma\in(0,1)$, where $\sigma$ may differ on the given meshes.

To obtain a geometric refinement in the FE subdomain $\Omega^1$ we
proceed as follows. 
Consider a coarse regular initial mesh $\meshFEMsigO$
consisting of triangles and parallelograms, such that for each corner 
$A_i,i=1,\dots,K$ of $\Omega^1$
with interior angle $\omega_i$ there are at most three parallelograms sharing
the common corner $A_i$, compare e.g. \cite[Section 4.5.3]{Swb98}. 
Denote by 
\begin{equation}\label{MeshSingPart}
	\hat{\frS}^i := \{K\in\meshFEMsigO: A_i\in \overline K\}
\end{equation}
the set of \emph{singular} elements containing the singularity at $A_i$ and for convenience we
also denote the complement, the set of \emph{regular} 
elements in $\meshFEMsigO$ that do not abut the singularity, by
\begin{equation}\label{MeshRegPart}
	\hat{\frR} := \{ K \in \meshFEMsigO: K\not\in \hat{\frS}\}
	\text{ with } \hat{\frS} := \cup_i \hat{\frS}^i.
\end{equation}
An initial geometric refinement with a fixed $\sigma\in(0,1)$
towards the singular vertex $A_i$ with 3 layers
(cf. Figure \ref{mesh:georef}) is now affinely mapped onto all $K\in\hat{\frS}^i$, $\forall i$.
Subsequently eliminating irregular nodes in $\meshFEMsigO$  by suitably 
refining or splitting some elements in $\hat{\frR}$, we construct a regular 
geometrically refined mesh $\meshFEMsig$.
Hence, we obtain a decomposition
\begin{equation}\label{SingMeshSplit}
	\meshFEMsig = \frS\cup \frR,\qquad \text{ with } \frS = \cup_i \frS^i.
\end{equation}
Moreover, all elements that arise from further geometric refinement of
elements in $\frS^i$ are contained in $\frS^i$. We set 
$\frS^i = \cup_{k=0}^n \frS^i_k$ and define the terminal layers 
$\frS^i_0$ and the $k$-th layers of $\frS^i$ by
\[
\begin{array}{rcl}
	\frS^i_0 &=& \{K\in\meshFEMsig: A_i\in\overline{K}\},\\
	\frS^i_k &=& \left\{K\in\meshFEMsig\backslash\left(\cup^{k-1}_{j=0} 
	\frS^i_j\right): \overline{K}\cap\overline{K^\prime}\neq\emptyset,
	K^\prime\in \frS^i_{k-1} \right\}.
\end{array}
\]
We collect the polynomial degrees of elements in the set 
$\mathbf{p} = \{p_K: K\in\meshFEMsig)$ and call $\mathbf{p}$ linear with slope $\mu>0$, if
for the elements in a mesh $\meshFEMsig$ with $n$ layers there holds 
$K\in\frS_j:=\cup_i \frS^i_j,j\geq 0:\: p_K = p_j,\: 
p_j =\max\{j+1,\lfloor\mu(j+1) \rfloor \}$
and $p_K=\max\{n+1,\lfloor\mu(n+1)\rfloor\}$ for $K\in\frR$, where the parameter
$\mu$ remains to be chosen in dependency of $\sigma$. 
\begin{figure}[t]
\centering
\begin{minipage}{0.45\textwidth}
\begin{center}
\begin{tikzpicture}[scale=4]
\draw[black, thick,-] (0,0) -- (0.125,0) -- (0.125,0.125) -- (0,0.125) -- (0,0);
\draw[black, thick,-] (0,0.125) -- (0,0.25) -- (0.125,0.25) -- (0.125,0.125); 
\draw[black, thick,-] (0.125,0) -- (0.25,0) -- (0.25,0.125) -- (0.125,0.125); 
\draw[black, thick,-] (0.125,0.25) -- (0.25,0.25) -- (0.25,0.125); 

\draw[black, thick,-] (0.25,0.125) -- (0.5,0) -- (0.25,0); 
\draw[black, thick,-] (0.25,0.125) -- (0.5,0.25) -- (0.5,0); 
\draw[black, thick,-] (0.25,0.25) -- (0.5,0.25) -- (0.5,0.5) -- (0.25,0.5) -- (0.25,0.25); 

\draw[black, thick,-] (0.125,0.25) -- (0.25,0.5) -- (0,0.5) -- (0.125,0.25); 
\draw[black, thick,-] (0,0.25) -- (0,0.5); 

\draw[black, thick,-] (0.5,0) -- (1,0) -- (0.5,0.25) -- (1,0.5) -- (1,0);
\draw[black, thick,-] (0.5,0.5) -- (1,0.5) -- (1,1) -- (0.5,1) -- (0,1) -- (0.25,0.5) -- (0.5,1) -- (0.5,0.5);
\draw[black, thick,-] (0,0.5) -- (0,1);
\end{tikzpicture}
\end{center}
\end{minipage}
\begin{minipage}{0.45\textwidth}
\begin{center}
\begin{tikzpicture}[scale=4]
\draw[black, thick,-] (0,0) -- (0.04,0) -- (0.04,0.04) -- (0,0.04) -- (0,0);
\draw[black, thick,-] (0,0.04) -- (0,0.2) -- (0.04,0.2) -- (0.04,0.04); 
\draw[black, thick,-] (0.04,0) -- (0.2,0) -- (0.2,0.04) -- (0.04,0.04); 
\draw[black, thick,-] (0.04,0.2) -- (0.2,0.2) -- (0.2,0.04); 

\draw[black, thick,-] (0.2,0.04) -- (1,0) -- (0.2,0); 
\draw[black, thick,-] (0.2,0.04) -- (1,0.2) -- (1,0); 
\draw[black, thick,-] (0.2,0.2) -- (1,0.2) -- (1,1) -- (0.2,1) -- (0.2,0.2); 

\draw[black, thick,-] (0.04,0.2) -- (0.2,1) -- (0,1) -- (0.04,0.2); 
\draw[black, thick,-] (0,0.2) -- (0,1); 
\end{tikzpicture}
\end{center}
\end{minipage}

\caption{Conforming geometric refinements on $[0,1]^2$ for a singularity in the lower left corner.  (left) refinement with 4 layers for $\sigma=0.5$,
(right) refinement with 3 layers for $\sigma=0.2$.} \label{mesh:georef}
\end{figure}
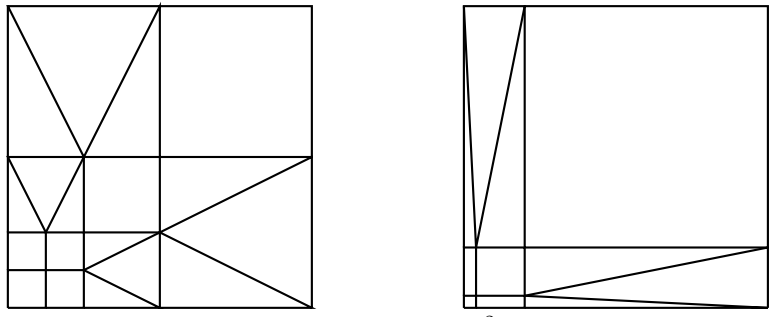

The construction of a geometrically refined mesh on $\Gamma^2$ is realised
by help of an affine mapping of a geometrically refined mesh on the interval 
$[0,1]$. More precisely, consider the geometric mesh 
$I^{\sigma,n}$ on $I:=[0,1]$ with $n$ layers $I_j = [x_{j-1},x_j], x_0 = 0, x_j 
= \sigma^{n+1-j}, h_j = x_j - x_{j-1}, 1\leq j\leq n+1$ and refinement parameter
$\sigma\in(0,1)$. For a linear degree vector $\mathbf{p}=(p_1,...,p_{n})$ with
slope $\mu >0 $ of nonnegative integers with 
$p_1 = 1, p_j = \max\{2,\lfloor \mu j\rfloor\},j=2,\dots,n$ let
\[
	V_{hp}^{\sigma,\mu}(I^{\sigma,n}) := 
	\{\Phi\in H^1(I): \Phi|_{I_j}\in\cP_{p_j}(I_j)\},
\]
where $\mu$ is chosen dependent on $\sigma$ in the proof for the 
$hp$-method.
Let $\Gamma^2$ be the union of $L$ straight arcs, i.e. 
$\Gamma^2 = \cup^L_{i=1} \Gamma^2_i$, and $\Gamma^2_j$ open straight arc with 
endpoints $B_{j},B_{j+1}\in\mathbb{R}^2$, where we use the convention that $B_{L+1} = B_1$. 
Each straight arc $\Gamma^2_j$ is now cut into two parts, such that 
$\Gamma^2_j = \Gamma^2_{j,1} \cup \Gamma^2_{j,2}$. 
Consequently, the geometric mesh on $I^{\sigma,n}$ is mapped to $\Gamma^2_{j,i}$ in such a
way that geometric refinement is carried out in direction of the vertices of 
$\Gamma^2_j$ on both subarcs.

The finite (resp. boundary) element spaces are then given by 
$\spaceFEMsig,\spaceFEMgradsig$ (resp. $\spaceBEMsig,\spaceBEMgradsig$) 
on $\meshFEMsig$ (resp. $\meshBEMsig$) and are defined in analogy to
\eqref{def:hpspaces} and \eqref{def:gradientspaces}. 
Similar to the quasi-uniform case we set 
\[\spaceFBEMsig :=  \spaceFEMsig \times \spaceBEMsig .\]
Furthermore, we denote by $\meshIone$ and $\meshItwo$ the trace meshes on 
$\Gamma_I$ induced by the partitions $\meshFEM$ and $\meshBEM$, respectively.
Analogously we define the trace meshes $\cT^{\sigma}_{hp}(\Gamma^1_I)$ and 
$\cT^{\sigma}_{hp}(\Gamma^2_I)$, which are induced by the partitions $\meshFEMsig$ and 
$\meshBEMsig$, respectively.

\begin{remark}
Note that in the following sections we state results on the space $V_{hp}$ and corresponding
gradient spaces $\mathbf{W}_{hp}(\Omega^1)$ and $W_{hp}(\Gamma^2)$. Unless specified
otherwise, these results also apply for the space $V^{\sigma,\mu}_{hp}$ and the
gradient spaces $\mathbf{W}^{\sigma,\mu}_{hp}(\Omega^1)$ and 
$W^{\sigma,\mu}_{hp}(\Gamma^2)$, respectively, as the parameters $\sigma$ and
$\mu$ only refer to geometric grading of the mesh and the growth of polynomial
degrees on the mesh elements.
\end{remark}
\noindent
As a necessity for our analysis we recall the countably normed spaces 
$\cB^l_{\beta}(\Omega^1)$ and $\cB^l_{\hat\beta}(\Gamma^2)$.
Let $\beta = (\beta_{1},\dots,\beta_{K})$ be an
$K$--tuple of real numbers with $\beta\in (0,1)^K$ 
and define the weight functions
\begin{equation}
\begin{split}
	\Phi^1_{\beta}(x) &= \prod^K_{i=1} 
	r_i(x)^{\beta_{i}},\quad r_i(x) = \min(1, \text{dist}(x,A_i)),\\
	\Phi^2_{\hat\beta}(x) &= \prod^2_{i=1} \hat{r}_i(x)^{\hat\beta_{i}},
	\quad \hat{r}_1(x) = |x| \text{ and } \hat{r}_2(x) = |x - 1|, 
	\hat\beta = (\hat\beta_{1}, \hat\beta_{2})\in(0,1)^2.
\end{split}	
\end{equation}
For integers $m\geq l\geq 0$ define the seminorms 
\begin{equation}
\begin{split}
	|u|^2_{H^{m,l}_{\beta}(\Omega^1)} &:= \sum\limits^m_{k=l} \|\,|D^k u|\,
	\Phi^1_{\beta+k-l}  \|^2_{L^2(\Omega^1)},\\
	|u|^2_{H^{m,l}_{\hat\beta}(I)} &:= \sum\limits^m_{k=l} \| u^{(k)}
	\Phi^2_{\hat\beta+k-l}  \|^2_{L^2(I)},
\end{split}
\end{equation}
where for 
$k\in\NN_0, \Phi^1_{\beta+k}(x) = \prod^K_{i=1} r_i(x)^{\beta_{i} + k}$ and
similarly for $\Phi^2_{\hat\beta + k}(x)$.
Moreover, by $H^{m,l}_{\beta}(\Omega^1)$ and  $H^{m,l}_{\hat\beta}(I),
m\geq l\geq 1,$ we denote the completion of 
$C^\infty(\overline{\Omega^1})$ and $C^{\infty}(\overline I)$
with respect to the norms
\begin{equation}
\begin{split}
		\| u \|^2_{H^{m,l}_{\beta}(\Omega^1)} &= \|u\|^2_{H^{l-1}(\Omega^1)} 
		+ |u|^2_{H^{m,l}_{\beta}(\Omega^1)},\\
		\|u\|^2_{H^{m,l}_{\hat\beta}(I)} &=  \|u\|^2_{H^{l-1}(I)} 
		+ |u|^2_{H^{m,l}_{\hat\beta}(I)}.
\end{split}
\end{equation}
\begin{definition}[Countably normed spaces]
 	Let $l\in\mathbb{N}_0$ and $0\leq\beta,\hat\beta<1$.
 	\begin{enumerate}
 	\item[(i)]{ For an open and bounded polygonal domain 
 	$\Omega^1\subset\mathbb{R}^2$ we define
 	\begin{equation}\label{defBlOmega}
 	\begin{split}
 	 \cB^l_{\beta}(\Omega^1) := \left\{u\in H^{m,l}_{\beta}(\Omega^1), 
 	 \forall m\geq l\geq 0,
	\|\Phi^1_{\beta+k-l}\,|D^k u|\|_{L^2(\Omega^1)} \leq C d^{k-l} (k-l)!,
	\right. \\
	k=l,l+1,...,\, C>0,d\geq 1 \text{ independent of } k\Big\}.
	\end{split}
	\end{equation} 
	with $|D^k u|^2 = \sum_{|\alpha|=k} |D^\alpha u|^2$ for 
	$\alpha\in\mathbb{N}^2_0$. }
	\item[(ii)]{ For $I=[0,1]$ we set
	\begin{equation}\label{defBlI}
	\begin{split}
	 \cB^l_{\hat\beta}(I) := \left\{ u\in H^{m,l}_{\hat\beta}(I), 
	 \forall m\geq l\geq 0,
	\|\Phi^2_{\hat\beta+k-l}\,|u^{(k)}| \|_{L^2(I)} 
	\leq C d^{k-l} (k-l)!,\right.
	 \\
	 k=l,l+1,...,\, C>0,d\geq 1 \text{ independent of } k \Big\}.
	\end{split}
	\end{equation}}
	\end{enumerate}
\end{definition}
\noindent For any $\Gamma^2_j\subset \Gamma^2$ the spaces 
$H^{k,l_j}_{\hat\beta_j}(\Gamma^2_j)$ and $B^{l_j}_{\hat\beta_j}(\Gamma^2_j)$ are
defined using a smooth map from $I$ to $\Gamma_j$ via the spaces 
$H^{k,l_j}_{\hat\beta_j}(I)$ and $B^{l_j}_{\hat\beta_j}(I)$. We then define 
$H^{k,l}_{\hat\beta}(\Gamma^2) := \prod^L_{j=1} H^{k,l_j}_{\hat\beta_j}(\Gamma^2_j)$
and 
$B^{l}_{\hat\beta}(\Gamma^2) := \prod^L_{j=1} B^{l_j}_{\hat\beta_j}(\Gamma^2_j)$
with $l = (l_1,\dots,l_L)$ and 
$\hat\beta = (\hat\beta_1,\dots,\hat\beta_L),\hat\beta_j = (\hat\beta_{j,1},
\hat{\beta_{j,2}})$. We write $\hat\beta_j\geq \tilde\beta_j$, if 
$\hat\beta_{j,k} \geq \tilde\beta_{j,k}, k=1,2,$ and 
$\hat\beta \geq \tilde\beta$ if $\hat\beta_j\geq\tilde\beta_j$, for all j.
Similarly we may write $\hat\beta_j \geq s$ for a real number $s$ if 
$\hat\beta_{j,k}\geq s, k=1,2$.

In general, for a bounded, open, and polygonal domain $\Omega$ with $\Gamma:=\partial \Omega$
the trace space of $H^{k,l}_\beta(\Omega)$ (resp. $\cB^l_\beta(\Omega)$) is
given by the space $H^{k-1/2, l-1/2}_\beta(\Gamma)$ 
(resp. $\cB^{l-1/2}_\beta(\Gamma)$), $k,l$ integral, $k\geq l\geq 0$, i.e. for
any $g\in H^{k-1/2, l-1/2}_\beta(\Gamma)$, (resp. $\cB^{l-1/2}_\beta(\Gamma)$)
there exists $G\in H^{k,l}_\beta(\Omega)$, 
(resp. $\cB^l_\beta(\Omega)$) such that $G|_\Gamma = g$, and
\[
	\|g\|_{H^{k-1/2, l-1/2}_\beta(\Gamma)} = \inf_{G|_\Gamma = g} 
	\|G\|_{H^{k,l}_\beta(\Omega)}.
\]
Generally speaking it is difficult to determine, if the data lie in these trace spaces.
The following result from \cite{BabGuoSteph90} and the references therein show
that the spaces $H^{k,l}_{\hat\beta}(\Gamma)$ 
(resp. $\cB^l_{\hat\beta}(\Gamma)$) characterize the traces of functions in
$H^{k,l}_\beta(\Omega)$ (resp. $\cB^l_\beta(\Omega)$) in a precisely verifiable
manner.
\begin{theorem}[\cite{BabGuoSteph90}]\label{thm:regtrace}
Let $u\in\mathcal{B}^l_\beta(\Omega)$ (resp. $H^{k,l}_\beta(\Omega)$,$\:k\geq l+1$), 
$l=1,2$, then for $1\leq i \leq M, u|_{\Gamma_i}\in 
\cB^{l-1}_{\hat\beta_i}(\Gamma_i)$ (resp. $H^{k,l-1}_{\hat\beta_i}(\Gamma_i)$)
with
\[ 
	\hat\beta_{i,j}\in(\beta_{i+j-1}-1/2,1/2), j=1,2 \text{ if } 
	1/2 <\beta_i,\beta_{i+1} < 1
\]
or $u|_{\Gamma_i}\in \cB^{l}_{\hat\beta_i}(\Gamma_i)$ 
(resp. $H^{k,l}_{\hat\beta_i}(\Gamma_i)$) with
\[
	\hat\beta_{i,j}\in(1/2,\beta_{i+j-1}+1/2), j=1,2 \text{ if } 
	0 <\beta_i,\beta_{i+1} < 1/2.
\]
\end{theorem}
\begin{remark}
Note the shift of the exponents $\hat{\beta}_{i,j}$ of the vertex weights in
dependence of $l$.
\end{remark}
\noindent
Alleviating notation we write $\cB^{l,l+1}_{\hat\beta}(\Gamma)$ to denote
the space of all functions whose restriction on $\Gamma_i$ belongs
to $\cB^l_{\hat\beta_i}(\Gamma_i), 0< \hat\beta_i< 1/2$ or 
$\cB^{l+1}_{\hat\beta_i}(\Gamma_i), 1/2 < \hat\beta_i < 1$, 
cf. also \cite{BabGuoSteph90}.

In the upcoming weak formulation we use a discretisation of the Steklov-Poincar\'e 
operator  $\cS$ (cf. \eqref{NSdef}), including $\cV^{-1}$, which cannot be
computed exactly. Therefore, we define an approximation 
\begin{equation}\label{Vm1hpdef}
\tilde\cV^{-1}_{hp} := i_{hp}(i^\prime_{hp}\cV i_{hp})^{-1} i^\prime_{hp},
\end{equation} 
where $i_{hp}: W_{hp}(\Gamma^2)\hookrightarrow H^{-1/2}(\Gamma^2)$ is the
canonical embedding, given as the natural injection, and its dual
$i^\prime_{hp}$, and write
\eqspl{hatSdef}
{
	\hat{\cS} := \cW + (\cK^\prime + 1/2) \tilde\cV^{-1}_{hp} (\cK + 1/2).
}
Properties of the boundary integral operators
involved then also yield continuity and coercivity of 
$\hat{\cS}$ on $V_{hp}(\Gamma^2)$ (the proof for variable $p$ is similar to
\cite[Lemma 12.11]{Stb08}):
\alil{coel_hatS}
{
	\exists c_{\hat{\cS}}>0:\dualp{\hat{\cS} u}{u}_{\Gamma^2} &
	\geq c_{\hat{\cS}} \norm{u}{H^{1/2}(\Gamma^2)}^2, &\forall& 
	u\in V_{hp}(\Gamma^2),\\
	\exists C_{\hat{\cS}}>0:\dualp{\hat{\cS} u}{v}_{\Gamma^2} &
	\leq C_{\hat{\cS}} \norm{u}{H^{1/2}(\Gamma^2)}^2 
	\norm{v}{H^{1/2}(\Gamma^2)}^2,
	&\forall& u,v\in V_{hp}(\Gamma^2).
}
Analogously an approximation is used for the Newton potential $\cN$ and write it as
\eqspl{hatNdef}
{
	\hat{\cN} := (\cK^\prime + 1/2)\tilde\cV^{-1}_{hp}\cN_0 - \cN_1.
}

Note, that the functions in the discrete space $\spaceFBEM$ 
(resp. $\spaceFBEMsig$) are in general discontinuous over $\Gamma_I$ and have
discontinuous normal derivatives. Moreover, continuity in general in the
strong sense cannot be imposed in case of non-matching meshes $\meshIone \not
\equiv \meshItwo$.
In this paper we construct a $hp$-Nitsche method, imposing continuity of the
solution $u$ in the weak sense. \\

\subsection{Relations to earlier work}

The $h$-version of Nitsche's finite element method was introduced and studied 
by Becker, Hansbo and Stenberg \cite{BeHaSg03}, with further
developments, e.g., by Hansbo, Hansbo and Larson \cite{HaHaLa03}, and
by Heinrich and co-workers \cite{HeJu06,HeNi03}. It can be treated as a 
mesh-dependent penalty method with additional terms, which, in contrast
to the original internal penalty methods (as proposed in \cite{Babu70}, with
recent contributions by Lazarov, Tomov and Vassilevski \cite{LazToVas01}), 
provide consistency of the coupling. In the context of BE, Nitsche's method has been analysed
for general boundary conditions by Betcke, Burman, and Scroggs \cite{betcke2019boundary} and
for FE/BE coupling by Betcke, Bosy, and Burman \cite{betcke2022} using a hybridised approach.

Stability of Nitsche's method is provided
by the stabilization function $\eta$, which should grow sufficiently fast 
with the mesh parameters.
A key point of the present analysis is that stability can be ensured with a stabilisation
parameter that depends locally on the mesh size and polynomial degree without assuming 
quasi-uniform meshes.
More precisely, we show that $\eta = \eta_0 \invrate$ is the optimal
choice, where $\invrate$ is the stability constant in the 
\textit{local} inverse inequality (\ref{InvRateConst}) and $\eta_0$ does not
depend on the discretisation parameters. This result is similar 
to one from Discontinuous Galerkin methods, cf. \cite{SchoeWih03}. Due to
consistency of the method there is no need to take larger
values of $\eta_0$ than required by stability. The penalty method, in contrast,
requires the penalty parameter to be of order $O(h^{-k})$, with $k\geq 1$
increasing with the polynomial degree, in order to control the consistency
error, cf. \cite{Babu70}. Using $k>1$, however, is detrimental to the
conditioning of the resulting discrete system.

We show in Theorem
\ref{thm:ba_bound} that stability and quasi-optimal convergence is achieved if 
$\eta_0 = (\alpha+2)\kappa$, where $\kappa$ is the diffusion coefficient and 
$\alpha>0$ is arbitrary. To our knowledge, we give the first sharp and local
estimate for $\eta$ which is not based on quasi-uniformity assumptions.
Furthermore, we give the first convergence analysis of Nitsche's method explicit
in \textit{local} polynomial degree, cf. Theorem \ref{thm:apriori_sharp} for
the a priori error estimate. The established error estimate is optimal in the
mesh size $h$ and suboptimal in the polynomial degree $p$ by the term $p^{1/2}$,
which is standard in the context of Discontinuous Galerkin methods, 
cf. \cite{HouSwbSueli02}.\\

An alternative method is the mortar method 
(see, e.g., Seshaiyer and Suri \cite{SeSu00} for $hp$-FEM). In the context of
the mortar method, the weak continuity of $u$ is enforced with the help of a
Lagrange multiplier $\lambda$, which leads to a saddle point formulation. It is
well known (see e.g. Ben Belgacem \cite{BBg99}, Wohlmuth \cite{Woh01}) that not
every discretisation of $u$ and $\lambda$ leads to a stable method. The 
{Babu\v{s}ka-Brezzi condition} is the crucial inequality which guarantees the
stable discretisation for the mortar method. 
Nitsche's method, on the other hand, leads to a positive definite system of
algebraic equations and is therefore always stable for $\eta_0$ large enough.
For a more detailed comparison between the mortar and Nitsche methods, see
Fritz, H\"ueber, and Wohlmuth \cite{FriHueWoh04}.\\

We use a symmetric boundary element formulation with the Steklov-Poincar\'e
operator $\cS$ (see e.g. Carstensen and Stephan \cite{CarSt95}) in the BE
subdomain. Since the operators $\cS$ and $\cN$ cannot be discretised directly, 
their approximations $\hat \cS$ and $\hat \cN$ are used, which yields a
consistency error, cf. Theorem \ref{thm:ConsErr}. This consistency 
error can be bounded by the approximation error of the discrete (traction)
boundary element space, which is optimal and does not disturb the 
convergence rate of the methods.

\section{The $hp$--Nitsche method}
In this section we introduce two formulations (\ref{weak_c}) and (\ref{weak_t})
of the $hp$--Nitsche method and discuss their basic properties. 
Identical over the discrete space $\spaceFBEM$, they are in general not
equivalent over $\fspace$. The formulation (\ref{weak_c}) 
is more convenient for practical implementation and (\ref{weak_t}) is the one
used in the proofs.

\subsection{Formulations}
Consider the following weak formulation:
find $U = (U^1,U^2) \in \spaceFBEM$, such that
\begin{equation} \label{weak_c}
\bilinc(U,\Phi)=l_{hp}(\Phi) \qquad \forall \Phi \in \spaceFBEM,
\end{equation}
where

\begin{align}
\bilinc(U,\Phi) := & (\kappa \nabla U^1, \nabla \Phi^1)_{\Omega^1} + \big<
\hat\cS U^2, \Phi^2 \big>_{\Gamma^2} \nonumber \\[2ex]
&- \big< q_\kappa(U),[\Phi] \big>_{\Gamma_I}
- \big<[U],q_\kappa(\Phi) \big>_{\Gamma_I}
+ \big< \eta[U], [\Phi]  \big>_{\Gamma_I},\label{bilincomp}\\[2ex]
l_{hp}(\Phi) := & (f, \Phi^1)_{\Omega^1} 
	+ \big<\hat \cN f, \Phi^2\big>_{\Gamma^2} 
	+ \big< g, \Phi \big>_{\Gamma_N},\label{linform}
\end{align}
where the stabilization function $\eta \in L^\infty(\Gamma_I)$ will be 
specified later. 
For the normal flux on the coupling interface we choose, as in \cite{HaHaLa03},
the one sided approximation from the FE subdomain
\[
q_\kappa(\Phi) := \kappa \nabla \Phi^1 \cdot \bn^1.
\]
Note that any convex combination 
$q_\kappa(\Phi) := \alpha \kappa \nabla \Phi^1 \cdot \bn^1 
+ (1-\alpha)\hat S \Phi^2$ might be used as well. The case  $\alpha \neq 1$ is,
however, less convenient from the practical point of view, since the operator 
$\hat \cS$ gives rise to additional dense blocks in the system matrix. 

Here $\hat \cS$ and $\hat \cN$ are the discrete Steklov-Poincar\'e operator
and the discrete Newton potential defined in (\ref{hatSdef}) and(\ref{hatNdef}),
respectively. This approximation introduces a consistency error, which is
standard for symmetric BEM \cite{Stb08}, cf. Theorem \ref{thm:ConsErr} below.\\

It turns out that (\ref{weak_c}) is not well suited for the error analysis of
the $hp$--Nitsche method, although it is appropriate for the error analysis of
the low $h$-version of Nitsche's method (i.e. with fixed polynomial degree $p$)
as in \cite{BeHaSg03,HaHaLa03}. The reason for this lies in the continuity
properties of $q_\kappa(\cdot)$ requiring analysis of the best approximation
error in the $H^2(\Omega^1)$-norm, which is not optimal in $\R^2$, 
see \cite{CsbDgeSwb05}.\\

Consider the lifting operator
$\lifting:V \to \spaceFEMgradstrip$ such that 
\begin{equation} \label{defLifting}
(\lifting(v),\bPsi)_{\Gstrip} = \big<[v],\bPsi \cdot \bn^1\big>_{\Gamma_I}
\qquad
\forall \bPsi \in \spaceFEMgrad.
\end{equation}
Note that for an element $K$ with $K\cap\Gamma_I=\emptyset$, i.e. an element not
adjacent to the interface, we have $\lifting(v)|_K=0$.
We define the bilinear form
\begin{equation} \label{bilintheo}
\begin{array}{rcl}
\bilint(U,\Phi) 
&:=& (\kappa \nabla U^1, \nabla \Phi^1)_{\Omega^1} 
    + \big<\hat \cS U^2, \Phi^2 \big>_{\Gamma^2} \\[2ex]
&& -(\kappa \nabla U^1,\lifting(\Phi))_{\Gstrip}
   -(\lifting(U),\kappa \nabla \Phi^1)_{\Gstrip} 
   + \big< \eta[U], [\Phi]  \big>_{\Gamma_I}.
\end{array}
\end{equation}
Note that $\bilinc(\Phi,\Psi) \equiv \bilint(\Phi,\Psi)$ for arbitrary 
$\Phi,\Psi \in \spaceFBEM$, since for $\Phi^1 \in \spaceFEM$ there holds 
$\nabla \Phi^1 \in \spaceFEMgrad$ and 
\[
\big<[\Psi], q_\kappa(\Phi) \big>_{\Gamma_I}
= \big<[\Psi],\kappa \nabla \Phi^1 \cdot \bn^1\big>_{\Gamma_I} 
= (\lifting(\Psi), \kappa\nabla \Phi^1)_{\Gstrip}.
\]
Thus, the discrete formulation (\ref{weak_c}) is equivalent to the problem of
finding $U \in \spaceFBEM$ such that
\begin{equation} \label{weak_t}
\bilint(U,\Phi) = l_{hp}(\Phi) \qquad \forall \Phi \in \spaceFBEM.
\end{equation}

\begin{remark}
Formulation (\ref{weak_c}) is more convenient for computations, since it does
not involve the lifting operator $\lifting$. On the other hand, formulation
(\ref{weak_t}) is better suited for the error analysis and will be used in the
forthcoming proofs.
\end{remark}

Suppose $\eta \in L^\infty(\Gamma_I)$ is a fixed piecewise constant function on
$\meshIone$
\[
0<\eta|_J = \eta_J<\eta_{\max},\qquad \forall J\in \meshIone.
\]
The following functional framework (cf. e.g. \cite{HaHaLa03,SchoeWih03}) is
required for the forthcoming error analysis.

Consider a $\eta$-dependent space 
$H^{1/2}_\eta(\Gamma_I)\subseteq L^2(\Gamma_I)$
endowed with the norm
\begin{equation} \label{def1/2norm} 
\|\phi\|_{\normp{\eta}} := \|\eta^{1/2}\phi\|_{L^2(\Gamma_I)} 
       = \left( \sum \limits_{J \in \meshIone}  \eta_J \|\phi\|^2_{L^2(J)}
       \right)^{1/2}
\end{equation}
and the space $H^{-1/2}_\eta(\Gamma_I)$ with the norm
\begin{equation} \label{def-1/2norm} 
\|\psi\|_{\normm{\eta}} =
\|\eta^{-1/2}\psi\|_{L^2(\Gamma_I)}
= \left( \sum \limits_{J \in \meshIone} \eta^{-1}_J \|\psi\|^2_{L^2(J)}
\right)^{1/2}.
\end{equation}
Note that the Cauchy-Schwarz inequality yields
\begin{equation}\label{eq:weightedCS}
\big<\phi, \psi \big>_{\Gamma_I} 
\leq \|\phi\|_{\normp{\eta}}\|\psi\|_{\normm{\eta}}.
\end{equation}

In what follows we analyse stability and convergence of the proposed 
$hp$-Nitsche method (\ref{weak_c}) in the following mesh-dependent norm, 
cf. also Remark \ref{rem:fullnorm},
\begin{align*}
\|\phi\|^2_\norme{\eta}
:= (\kappa \nabla \phi^1, \nabla \phi^1)_{\Omega^1}
+ \big<\hat \cS \phi^2, \phi^2 \big>_{\Gamma^2}
+ \|[\phi]\|^2_{\normp{\eta}} .
\end{align*}

Later on (in Theorem \ref{thm:ba_bound}) we will define the weight function 
$\eta = \eta(h,p)$ in a natural way to ensure stability and convergence.
Moreover, for $s> 0$ we define the spaces of piecewise functions 
$H^{s}_{p.w.}(\Gamma^2) := \{v \in L^2(\Gamma^2): v|_J \in H^s(J), \forall J 
\in \meshBEM\}$ endowed with the broken norm
\[
\|\phi\|_{H^{s}_{p.w.}(\Gamma^2)}^2
:= \sum \limits_{J \in \meshBEM} \|\phi\|_{H^s(J)}^2.
\]

\subsection{Consistency of the modified discrete formulation} 
Let us introduce for arbitrary $v \in \fspace$, $\Phi \in \spaceFBEM$ and  
$\eta \in L^\infty(\Gamma_I)$, $\eta>0$ the residual
\begin{align} 
\res{\eta}(v) := \sup_{0 \neq \Phi \in \spaceFBEM}
\dfrac{\resphi(v,\Phi)}{\|\Phi\|_\norme{\eta}}, \quad
\mbox{ where }
\resphi(v,\Phi) := \bilint(v,\Phi) - l_{hp}(\Phi).
\label{def_res}
\end{align}
The following theorem allows to control the residual of the exact solution
$u \in \fspace$. 
\begin{theorem}[Representation of Residual] 
\label{thm:ConsErr}
Let $u \in \fspace$ be the weak solution of (\ref{p2}) and
let $\bPi_{hp}:[L^2(\Gstrip)]^2 \to \spaceFEMgradstrip$ be the $L^2$-projection.
Then
\begin{equation}\label{eq:resrep}
\resphi(u,\Phi)
= \big< \kappa (\nabla u^1 - \bPi_{hp}( \nabla u^1))\cdot \bn^1,[\Phi]
\big>_{\Gamma_I}
- \big<(\cS - \hat \cS) u^2 - (\cN - \hat \cN) f^2, \Phi^2 \big>_{\Gamma^2}.
\end{equation}
\end{theorem}
\begin{proof}
For $u\in \fspace$ solving (\ref{p2}) there holds
\begin{equation} \label{jumps}
[u]=0, \qquad
q_\kappa(u) = \kappa \nabla u^1 \cdot \bn^1 = -\nabla u^2 \cdot \bn^2 
\qquad \mbox{ a.e. on } \Gamma_I.
\end{equation}
Therefore for $\forall \Phi \in \spaceFBEM$
\begin{equation*}
\begin{array}{rcl}
\bilint(u,\Phi) &=& (\kappa \nabla u^1, \nabla \Phi^1)_{\Omega^1} 
+ \big<\hat \cS u^2, \Phi^2 \big>_{\Gamma^2} \\[2ex]
&&- ( \kappa \nabla u^1,\lifting(\Phi) )_{\Gstrip}
  - (\lifting(u),\kappa \nabla \Phi^1 )_{\Gstrip}
+ \big< \eta[u], [\Phi]  \big>_{\Gamma_I} \\[2ex]
&=& (\kappa\nabla u^1, \nabla \Phi^1)_{\Omega^1} 
+ \big<\hat \cS u^2, \Phi^2 \big>_{\Gamma^2}
- ( \kappa \nabla u^1,\lifting(\Phi) )_{\Gstrip}
\end{array}
\end{equation*}
Recall Green's formula (\ref{Green})
\[
(\nabla u^2, \nabla \Phi^2)_{\Omega^2} - (f, \Phi^2)_{\Omega^2} =
\big<\cS u^2 - \cN f, \Phi^2\big>_{\Gamma^2},
\qquad \forall \Phi^2 \in H^1(\Omega^2).
\]
Hence, with 
partial integration in $\Omega^1$ and $\Omega^2$, respectively, we obtain
\begin{align*}
\resphi(u,\Phi)=&\bilint(u,\Phi) - l_{hp}(\Phi)\\[2ex]
=&(\kappa \nabla u^1, \nabla \Phi^1)_{\Omega^1} 
+(\nabla u^2, \nabla \Phi^2)_{\Omega^2} 
- (f, \Phi)_{\Omega}
- \big< g, \Phi \big>_{\Gamma_N}\\[2ex]
&
- (\kappa \nabla u^1,\lifting(\Phi))_{\Gstrip}
- \big<(\cS - \hat \cS) u^2, \Phi^2 \big>_{\Gamma^2}
+ \big<(\cN - \hat \cN) f^2, \Phi^2 \big>_{\Gamma^2}\\[2ex]
=& 
 \big<\kappa \nabla u^1 \cdot \bn^1 , [\Phi] \big>_{\Gamma_I}
- ( \kappa \nabla u^1,\lifting(\Phi))_{\Gstrip}\\[2ex]
&- \big<(\cS - \hat \cS) u^2, \Phi^2 \big>_{\Gamma^2}
 + \big<(\cN - \hat \cN) f^2, \Phi^2 \big>_{\Gamma^2}
\end{align*}
since $u \in \fspace$ is a weak solution of (\ref{p2}) and (\ref{jumps}) holds
true. Recall that $\bPi_{hp}:\big[L^2(\Gstrip)\big]^2 \to \spaceFEMgradstrip$
is the $L^2$-projection. Then
\[
( \kappa \nabla u^1,\lifting(\Phi))_{\Gstrip}
= ( \kappa \,\bPi_{hp}( \nabla u^1),\lifting(\Phi))_{\Gstrip}
= \big< \kappa\,\bPi_{hp}( \nabla u^1)\cdot \bn^1,[\Phi]\big>_{\Gamma_I},
\]
since $\lifting(\Phi) \in \spaceFEMgradstrip$,
yielding
\begin{equation*}
\resphi(u,\Phi)
= \big< \kappa (\nabla u^1 - \bPi_{hp}( \nabla u^1))\cdot \bn^1,
[\Phi]\big>_{\Gamma_I}
- \big<(\cS - \hat \cS) u^2 - (\cN - \hat \cN) f^2, \Phi^2 \big>_{\Gamma^2}.
\end{equation*}
\end{proof}
\begin{corollary}[Consistency error]\label{cor:conserr}
Let $u \in \fspace$ be the weak solution of (\ref{p2}) and
let $\bPi_{hp}:[L^2(\Gstrip)]^2 \to \spaceFEMgradstrip$ be the $L^2$-projection.
Then
\begin{equation} \label{res_bound}
\begin{array}{l}
\res{\eta}(u)
 \leq 
\|\kappa(\nabla u^1 - \bPi_{hp}(\nabla u^1))\|_\normm{\eta}
+ c_{\hat \cS}^{-1/2}
\|(\cS - \hat \cS) u^2 - (\cN - \hat \cN) f^2\|_{H^{-1/2}(\Gamma^2)}
\end{array}
\end{equation}
where $c_{\hat \cS}$ is the coercivity constant of $\hat \cS$, 
cf. (\ref{coel_hatS}).
\end{corollary}
\begin{proof}
Estimate (\ref{res_bound}) follows from \eqref{eq:resrep} by Cauchy-Schwarz
inequality \eqref{eq:weightedCS} and the estimate
\[
\|\Phi^2\|_{H^{1/2}(\Gamma^2)} \leq c_{\hat{\cS}}^{-1/2} \|\Phi\|_\norme{\eta}.
\]
\end{proof}

\begin{remark}
We note that in order to control the first term in \eqref{eq:resrep} it is
sufficient in the quasi-uniform case that 
$u^1\in H^{3/2+\varepsilon}(\Omega^1)$. In the case that 
$u\in\cB^2_{\beta}(\Omega^1)$ (resp. $H^{2,2}_{\beta}(\Omega^1)$) the first
term in \eqref{eq:resrep} cannot be bounded directly on elements 
$K\in\frS_0, K\cap\Gamma_I\neq \emptyset$ using the Cauchy-Schwarz inequality. 
A slightly more involved analysis is neccessary, since 
$\nabla u^1 \in [\cB^1_{\beta}(\Omega^1)]^2$ 
(resp. $[H^{1,1}_{\beta}(\Omega^1)]^2$), 
cf. Theorem \ref{thm:apriori_exp} below.
\end{remark}

The following Lemma provides a useful estimate of the
consistency error caused by the approximation of boundary integral operators,
which constitutes the second term in \eqref{eq:resrep}.
\begin{lemma}\label{lem:stekcons}
Let $u=(u^1,u^2)\in V$
solve \eqref{p2} and $f=(f^1,f^2)$ be a sufficiently regular right-hand side
datum. Then there holds
\begin{equation}\label{eq:steklovres}
	\|(\cS-\hat\cS)u^2 - (\cN - \hat\cN) f^2\|_{H^{-1/2}(\Gamma^2)} \leq 
	C \inf_{\psi_{hp}\in W_{hp}(\Gamma^2)} 
	\|\tfrac{\p u^2}{\p \bn^2}-\psi_{hp}\|_{H^{-1/2}(\Gamma^2)}.
\end{equation}
\end{lemma}
\begin{proof} We prove the assertion for the interior problem. For exterior
problems the arguments are analogous. By inserting the first line of Calder\`on
system \eqref{eq:intCalderon} into its second line,
we have the following representation of the normal derivative of the solution to
the model problem on $\Gamma^2$
\begin{align*}
	\frac{\partial u^2}{\partial \bn^2} &= \mathcal{W} u^2 + 
	\left( \mathcal{K}' + \tfrac{1}{2} \right) 
	\frac{\partial u^2}{\partial \bn^2} 
	- \mathcal{N}_1 f^2\\
	&= \mathcal{W} u^2 + (\mathcal{K}'+\tfrac{1}{2})
	\mathcal{V}^{-1}(\mathcal{K}+\tfrac{1}{2}) u^2
	- ((\mathcal{K}'+\tfrac{1}{2})\mathcal{V}^{-1}\mathcal{N}_0 f^2 +
	\mathcal{N}_1f^2) = \cS u^2 - \cN f^2,
\end{align*}
whence with \eqref{Vm1hpdef}, \eqref{hatSdef}, and \eqref{hatNdef} we deduce
\[
	(\cS-\hat\cS)u^2 -(\cN-\hat\cN)f^2 = (\mathcal{K}' + 
	\tfrac{1}{2})(\cV^{-1} - \tilde\cV^{-1}_{hp}) 
	((\mathcal{K}+\tfrac{1}{2})u^2 - \mathcal{N}_0f^2)).
\]
Taking the $H^{-1/2}$-norm over $\Gamma^2$, we have by continuity of the
operator $(\mathcal{K}' + \frac{1}{2})$ on $H^{-1/2}(\Gamma^2)$, cf. e.g.
\cite{Stb08}, that
\begin{align*}
	\|(\cS-\hat\cS)u^2 - (\cN - \hat\cN) f^2\|_{H^{-1/2}(\Gamma^2)} 
	&\leq C \|(\cV^{-1} - \tilde\cV^{-1}_{hp}) 
	((\mathcal{K}+\frac{1}{2})u^2 - \mathcal{N}_0 f^2)\|_{H^{-1/2}(\Gamma^2)} \\
	&= C \|z - z_{hp}\|_{H^{-1/2}(\Gamma^2)},
\end{align*}
where $z = \p u^2 / \p  \bn^2$ and its $hp$-approximant 
$z_{hp}\in W_{hp}(\Gamma^2)$ (resp. $W^{\sigma,\mu}_{hp}(\Gamma^2)$) by the first line
of \eqref{eq:intCalderon}. Since the operator $\mathcal{V}$ 
is strongly elliptic on $H^{-1/2}(\Gamma^2)$ (cf. \cite{CosSteph85}), 
i.e. satisfies a G\aa rding inequality, 
any conforming Galerkin scheme converges, and there holds
\[
	\| z - z_{hp}\|_{H^{-1/2}(\Gamma^2)} \leq  C \inf_{\psi_{hp}\in
	 W_{hp}(\Gamma^2)} 
	 \|\tfrac{\p u^2}{\p \bn^2} - \psi_{hp}\|_{H^{-1/2}(\Gamma^2)}.
\]
\end{proof}

\subsection{Inverse inequality and stability of the lifting operator}

The following Lemma is important for proving coercivity of 
$\bilint(\cdot,\cdot)$.

\begin{lemma}[Polynomial trace inequality] \label{lem:InvIneq}
Let $K = F_K(\hat K)$ be an affine image of a reference element $\hat K$, i.e.
either $\hat K = \hat T$ or $\hat K = \hat Q$ and 
$\Psi \circ F_K\in R_p(\hat K)$, so that $\Psi$ is a polynomial on $K$. 
Suppose $J$ is an edge of $K$. Then
\begin{equation}\label{InvIneq}
\|\Psi\|_{L^2(J)}^2
\leq \invrate_K \|\Psi\|_{L^2(K)}^2,
\end{equation}
where 
\begin{equation}\label{InvRateConst}
\invrate_K =  
\begin{cases}
\frac{(p+1)(p+2)}{2} \frac{\text{\emph{Volume}}(J)}{\text{\emph{Volume}(K)}}, &
\text{if} \:\: \hat K = \hat T,\\[2ex]
(p+1)^2  \frac{\text{\emph{Volume}(J)}}{\text{\emph{Volume}(K)}}, 
& \text{if} \:\: \hat K = \hat Q.
\end{cases}
\end{equation}
\end{lemma}
\begin{proof}
The proof for $\hat K = \hat T$ is given in \cite[Theorem 5]{WrbHst03}. The case
of $\hat K = \hat Q$  is proved using a tensor argument in combination with 
\cite[Theorem 2]{WrbHst03}.
\end{proof}

\begin{remark}
Note that 
\begin{equation}\label{invrate_est}
\invrate_K \sim h_K^{-1}(p_{K}+1)^2,
\end{equation}
where the equivalence constant only depends on the shape regularity of $K$.
\end{remark}

As a corollary of Lemma \ref{lem:InvIneq} we get the following stability
estimate for the lifting operator.
\begin{lemma} \label{lem:L_stab}
For arbitrary $v \in \fspace$ there holds
\begin{equation}\label{L_stab}
\|\kappa^{1/2} \lifting(v)\|_{[L^2(\Gstrip)]^2}
\leq \|[v]\|_\normp{\kappa \invrate},
\end{equation}
where $\invrate \in L^\infty(\Gstrip)$ is a constant $\invrate_K$ on 
$K \in \meshFEM$ as in (\ref{InvRateConst}).
\end{lemma}

\begin{proof}
We use the orthogonal decomposition $\big[L^2(K)\big]^2 = 
\spaceFEMgradK \oplus \spaceFEMgradK^\perp$ to obtain for every 
$K \in \meshFEM$ 
with $K\cap \Gamma_I\neq\emptyset$
\begin{align*}
\|\lifting(v)\|_{[L^2(K)]^2}
= \sup_{0\neq \bPhi \in [L^2(K)]^2} 
\dfrac{(\lifting(v),\bPhi)_K}{\|\bPhi\|_{[L^2(K)]^2}}
= \sup_{0\neq \bPhi \in \spaceFEMgradK} 
\dfrac{(\lifting(v),\bPhi)_K}{\|\bPhi\|_{[L^2(K)]^2}}.
\end{align*}
Recalling definition (\ref{defLifting}) we find for $J_K := K \cap \Gamma_I$
\[
(\lifting(v),\bPhi)_K
= \big<[v],\bPhi\cdot \bn^1 \big>_{J_K}
\leq \|[v]\|_{L^2(J_K)} \|\bPhi \cdot \bn^1\|_{L^2(J_K)}
\leq \|[v]\|_{L^2(J_K)} \invrate_K^{1/2} \|\bPhi\|_{[L^2(K)]^2},
\]
where the inverse inequality (\ref{InvIneq}) has been used with the polynomial  
$\Psi^2 := \bPhi \cdot \bPhi$.
Summing over the elements gives
\begin{equation*}
\|\kappa^{1/2}\lifting(v)\|^2_{[L^2(\Gstrip)]^2}
\leq \sum_{K \in \meshFEM} \kappa_K \invrate_K \|[v]\|_{L^2(J_K)}^2
= \|[v]\|_\normp{\kappa \invrate}^2.
\end{equation*}
\end{proof}

\subsection{Continuity and coercivity of $\bilint$, quasi-optimal convergence}

We next show the continuity and coercivity of the bilinear form 
$\bilint(\cdot,\cdot)$ that is required to ensure convergence of the proposed
method.

\begin{lemma}[Continuity of $\bilint$] \label{lem:Cont_sharp} 
There holds
\begin{equation}\label{Cont_sharp}
\bilint(v,\Psi) \leq 2
\|v   \|_\norme{(\eta+\kappa \invrate)}
\|\Psi\|_\norme{(\eta+\kappa \invrate)}, 
\qquad \forall v \in \fspace,\quad \Psi \in \spaceFBEM.
\end{equation}
\end{lemma}

\begin{proof}
The Cauchy-Schwarz inequality gives for arbitrary $v \in \fspace$, 
$\Psi \in \spaceFBEM$
\begin{align*}
\bilint(v,\Psi)
\leq& 
\left(
2 (\kappa \nabla v^1,\nabla v^1)_{\Omega^1}
+ \big<\hat S v^2,v^2\big>_{\Gamma^2}
+ \|\kappa^{1/2}\lifting(v)\|_{[L^2(\Gstrip)]^2}^2
+ \|[v]\|_\normp{\eta}^2
\right)^{1/2}\\
\times& 
\left(
2(\kappa \nabla \Psi^1,\nabla \Psi^1)_{\Omega^1} 
+ \big<\hat S  \Psi^2,\Psi^2\big>_{\Gamma^2}
+ \|\kappa^{1/2}\lifting(\Psi)\|_{[L^2(\Gstrip)]^2}^2
+ \|[\Psi]\|_\normp{\eta}^2
\right)^{1/2}.
\end{align*}
The assertion follows by stability of $\lifting$ (\ref{L_stab}) and
\begin{align*}
\|\kappa^{1/2}\lifting(v)\|_{[L^2(\Gstrip)]^2}^2
+ \|[v]\|_\normp{\eta}^2
\leq
\sum_{J \in \meshIone}  \|(\eta+\kappa \invrate)^{1/2}[v]\|_{L^2(J)}^2.
\end{align*}
\end{proof}

\begin{remark} \label{rem:fullnorm}
In formulation (\ref{weak_c}), continuity only holds with respect to the
extended norm 
\[
\interleave\phi \interleave^2_\norme{\eta}:=\|\phi\|^2_\norme{\eta} +
\|\kappa\nabla\phi^1\cdot\mathbf {n^1}\|^2_\normm{\eta},
\]
since the traces of the normal derivatives  cannot be bounded 
on $V.$ Note, however, that $\interleave\cdot \interleave_\norme{\eta}$ and 
$\|\cdot\|_\norme{\eta}$ are equivalent norms on 
$V_{hp}$ due to (\ref{InvIneq}).
\end{remark}

\begin{lemma} [Coercivity of $\bilint$] \label{lem:ell} The bilinear form 
$\bilint$ is coercive in 
$({\spaceFBEM},{\|\cdot\|_\norme{(\eta-\delta\kappa \invrate)}})$ provided
\begin{equation*} \label{sigma_stab}
\eta > \delta\kappa \invrate, \qquad
\delta >1 \quad  
\mbox{ a.e. on } \Gamma_I,
\end{equation*}
i.e.
\begin{equation}\label{ell}
\bilint(\Psi,\Psi) 
\geq \dfrac{\delta-1}{\delta}\|\Psi\|_\norme{(\eta - \delta\kappa \invrate)}^2,
\qquad \forall \Psi \in \spaceFBEM.
\end{equation}
\end{lemma}
\begin{proof}
By definition (\ref{bilintheo}) of the bilinear form $\bilint$ we have
\begin{equation*}
\bilint(\Psi,\Psi) \\
= (\kappa \nabla \Psi^1, \nabla \Psi^1)_{\Omega^1} 
+ \big<\hat \cS \Psi^2, \Psi^2 \big>_{\Gamma^2}
- 2(\kappa \nabla \Psi^1,\lifting(\Psi) )_{\Gstrip}
+ \|[\Psi]\|^2_{\normp{\eta}}.
\end{equation*}
The stability of $\lifting$ expressed by (\ref{L_stab}) ensures that for
arbitrary $\delta > 0$ there holds
\begin{align*}
2(\kappa \nabla \Psi^1,\lifting(\Psi) )_{\Gstrip}
&\leq \delta^{-1} 
(\kappa \nabla \Psi^1,\nabla \Psi^1)_{\Gstrip} 
+ \delta \|\kappa^{1/2}\lifting(\Psi)\|^2_{[L^2(\Gstrip)]^2}\\
&\leq \delta^{-1}
(\kappa\nabla\Psi^1, \nabla\Psi^1)_{\Omega^1} 
+ \delta \|[\Psi]\|^2_{\normp{\kappa\invrate}}.
\end{align*}
Thus,
\begin{equation*}
\bilint(\Psi,\Psi) \\
\geq (1 - \delta^{-1}) 
(\kappa\nabla \Psi^1,\nabla\Psi^1)_{\Omega^1} 
+ \big<\hat \cS \Psi^2, \Psi^2 \big>_{\Gamma^2}
+ \|[\Psi]\|^2_{\normp{(\eta-\delta\kappa\invrate)}}
\end{equation*}
and the assertion follows.

\end{proof}

Note that the norms $\|\cdot\|_\norme{(\eta+\kappa \invrate)}$ and 
$\|\cdot\|_\norme{(\eta - \delta\kappa \invrate)}$ are equivalent on the finite
dimensional space $\spaceFBEM$ and the equivalence constants are independent of
$h, p$ if $\eta = (\delta+\varepsilon)\kappa \invrate$ for some $\delta>1$ and 
$\varepsilon > 0$, cf. Theorem \ref{thm:ba_bound} below.
The right hand side $l_{hp}(\cdot)$ in the discrete formulation (\ref{weak_t})
is continuous on $\spaceFBEM$. 
The bilinear form $\bilint(\cdot,\cdot)$, due to Lemma \ref{lem:Cont_sharp} and
Lemma \ref{lem:ell}, is continuous and coercive on $\spaceFBEM$.
Thus the Lax-Milgram lemma guarantees
\begin{theorem}
The discrete problem (\ref{weak_t}) has a unique solution $U \in \spaceFBEM$.
\end{theorem}

Based on the previous estimates we derive quasi-optimal convergence of the
proposed methods in the $\|\cdot\|_\norme{\kappa \invrate}$-norm for
sufficiently large $\eta$.

\begin{theorem}[Quasi-optimal convergence]\label{thm:ba_bound}
Suppose $u \in \fspace$ is a weak solution of (\ref{p2}) and $U \in \spaceFBEM$
is a solution of (\ref{weak_c}) or (\ref{weak_t}) with 
the stabilization function $\eta := \eta_0\kappa \invrate$ for arbitrary but
fixed constant $\eta_0>\delta>1$. Then 
\begin{equation} \label{ba_bound}
\|u - U\|_\norme{\kappa \invrate} \leq (2 c_{\delta,\eta_0} +1)
\inf_{\Phi \in \spaceFBEM}
\|u - \Phi\|_\norme{\kappa \invrate}
+ c_{\delta,\eta_0} \res{\kappa \invrate}(u)
\end{equation}
with $c_{\delta,\eta_0} = 
\dfrac{\delta}{\delta-1}\cdot\dfrac{\eta_0 + 1}{\min\{\eta_0-\delta,1\}}$.
\end{theorem}
\begin{proof}
There holds
\begin{align*}
\frac{\delta-1}{\delta} \|U - \Phi\|_\norme{(\eta-\delta\kappa\invrate)}^2
\stackrel{(\ref{ell})}{\leq}& 
  \bilint(U-\Phi,U-\Phi) \\[1ex]
\stackrel{(\ref{def_res})}{=}&
  \bilint(u-\Phi,U-\Phi)
 - \resphi(u,U-\Phi)\\[1ex]
\overset{(\ref{Cont_sharp}),(\ref{def_res})}{\:\,\,\quad\leq}&
\big(2\|u - \Phi\|_\norme{(\eta+\kappa \invrate)} +
 \res{(\eta+\kappa \invrate)}(u)\big)
 \|U - \Phi\|_\norme{(\eta+\kappa \invrate)}.
\end{align*}
Thus, for $\eta = \eta_0\kappa \invrate$ with $\eta_0>\delta$
\[
\|U-\Phi\|_\norme{\kappa\invrate}
\leq 
2c_{\delta,\eta_0}\|u - \Phi\|_\norme{\kappa \invrate} 
+ c_{\delta,\eta_0}\res{\kappa \invrate}(u).
\]
Therefore by triangle inequality
\begin{align*}
\|u-U\|_\norme{\kappa\invrate} &\leq \|u- \Phi\|_\norme{\kappa \invrate} 
+ \|U- \Phi\|_\norme{\kappa \invrate}\\[1ex]
&\leq (2c_{\delta,\eta_0}+1)\|u - \Phi\|_\norme{\kappa \invrate} 
+ c_{\delta,\eta_0}\res{\kappa\invrate}(u)
\end{align*}
and passing to the infimum over all $\Phi\in\spaceFBEM$ proves the claim.
\end{proof}

\begin{remark}
Note that 
\[
\min_{\delta>1,\:\: \eta_0-\delta>0} c_{\delta,\eta_0} =  2 (2+\sqrt{3})
\]
is attained at $\delta = 1+\sqrt{3}$, $\eta_0-\delta = 1$.
\end{remark}

\section{A priori error analysis}

\indent Theorem \ref{thm:ba_bound} allows us to bound 
$\|u-U\|_\norme{\kappa \invrate}$ by the sum of the best approximation terms on
the right-hand side of (\ref{ba_bound}). In what follows we recall some
approximation properties and deduce a priori error estimates in Theorem
\ref{thm:apriori_sharp} and Theorem \ref{thm:apriori_exp}. 

\subsection{$hp$-Nitsche on quasi uniform meshes}
\begin{lemma} \label{lem:estinf1}
Suppose $v \in \fspaceFEM \cap H^r(\Omega^1)$, $r > 3/2$. Then 
$\exists \Psi \in \spaceFEM$ and $C>0$ independent of the discretisation
parameters such that
\begin{align}
\|\nabla(v - \Psi)\|_{L^2(K)}
&\leq C 
\dfrac{h_K^{\rho_K - 1}}{p_K^{r-1}}
\|v\|_{H^r(K)},&&
\forall K \in \meshFEM, \label{estinf11}\\
\|v- \Psi\|_{L^2(J)} 
&\leq C \dfrac{h_K^{\rho_K - 1/2}}{p_K^{r-1/2}}
\|v\|_{H^r(K_J)},&&
\forall J \in \meshIone,\label{estinf12}
\end{align}
where $\rho_K := \min\{r,p_K+1\}$ and $K_J \in \meshFEM$ is such that $J$ is a
side of $K_J$.
\end{lemma}
\begin{proof}
Estimate (\ref{estinf11}) follows directly from \cite[Lemma 4.5]{BSu87hp}.
Estimate (\ref{estinf12}) follows directly from  \cite[Lemma 4.5]{BSu87hp},
since $J \subset \Gamma_I \subset \partial \Omega^1$ and no polynomial trace
lifting is necessary. Note that the estimate (\ref{estinf12}) in 
\cite[Lemma 4.5]{BSu87hp} has a printing error.
\end{proof}

\begin{lemma} \label{lem:estinf2}
Let $v \in \fspaceBEM \cap H^{r}_{p.w.}(\Gamma^2)$, $r>1$. Then 
$\exists \Psi \in \spaceBEM$ and $C > 0$ independent of the discretisation
parameters, such that
\begin{align}
\|v - \Psi\|_{L^2(J)} 
&\leq C \dfrac{h_J^{\rho_J}}{p_J^r} \|v\|_{H^r(J)}&&
\forall J \in \meshBEM, \label{estinf21}\\
\|v - \Psi\|_{H^{1/2}(\Gamma^2)} 
&\leq C \max_{J \in \meshBEM}\left\{{\dfrac{h_J^{\rho_J-1/2}}{p_J^{r-1/2}}}
\right\}\|v\|_{H^r_{p.w.}(\Gamma^2)}\label{estinf22}
&&
\end{align}
where $\rho_J := \min\{r,p_J+1\}$.
\end{lemma}
\begin{proof}
Choosing $\Psi$ as the quasi-interpolation operator from 
\cite[Lemma 4.3]{BSu87hp}, \cite[Theorem 3.17]{Swb98} we obtain (\ref{estinf21})
and
\begin{equation} \label{estinf21_}
\|v - \Psi\|_{H^1(J)} 
\leq C \dfrac{h_J^{\rho_J-1}}{p_J^{r-1}} \|v\|_{H^r(J)} \qquad 
\forall J \in \meshBEM.
\end{equation}
Thus by real method of interpolation
$
\|v - \Psi\|_{H^{1/2}(\Gamma^2)}^2 \leq
\|v - \Psi\|_{L^2(\Gamma^2)}
\|v - \Psi\|_{H^1(\Gamma^2)}
$
and the right-hand side may be localized to $J \in \meshBEM$. Then we use
(\ref{estinf21}), (\ref{estinf21_}).
\end{proof}

\begin{lemma} \label{lem:estT3}
Suppose $v \in H^r(\Gstrip)$, $r > 1/2$ and 
$\Pi_{hp}: \Gstrip \to \spaceFEMgradstripcomp$ be the (elementwise) 
$L^2$-projection. Then there exists $C >0$ independent of the discretisation
parameters, such that
\begin{equation} \label{estT3}
\|v - \Pi_{hp}v\|_{L^2(J)} \leq 
C \dfrac{h_J^{\rho_J-1/2}}{p_J^{r-1/2}}
\|v\|_{H^r(K_J)}, \qquad \forall J \in \meshIone
\end{equation}
where $\rho_J := \min\{r,p_J+1\}$ and $K_J \in \meshFEM$ is such that $J$ is a
side of $K_J$.
\end{lemma}
\begin{proof}
Note that $K \in \Gstrip$ is either a parallelogram and $\Psi|_K$ is a polynomial
of degree $p_J$ in the direction of the sides of $K$. Then (\ref{estT3}) follows
from \cite[Lemma 4.1, Remark 4.3]{GeoHalMel09} by scaling and a tensor product
argument. Otherwise $K$ is a triangle, in which case we get the estimate by
virtue of \cite[Corollary 1.2]{MelWur14} and usual translation and dilation
arguments.
\end{proof}

\begin{lemma}\label{lem:estT45}
Suppose $v \in L^2(\Gamma^2)\cap H^r_{p.w.}(\Gamma^2)$, $r\geq 0$ and 
$\Pi_{hp}: L^2(\Gamma^2) \to \spaceBEMtrac$ be the (elementwise) 
$L^2$-projection. Then there exists $C >0$ independent of the discretisation
parameters, such that
\begin{equation} \label{estT45}
\|v-\Pi_{hp} v\|_{H^{-1/2}(\Gamma^2)}
\leq C \max_{J \in \meshBEM}\left\{
{\dfrac{h_J^{\rho_J+1/2}}{p_J^{r+1/2}}}\right\}\|v\|_{H^r_{p.w.}(\Gamma^2)}
\end{equation}
where $\rho_J := \min\{r,p_J+1\}$.
\end{lemma}
\begin{proof}
The estimate follows from (\ref{estinf21}) and the duality argument:
\begin{align*}
&\|v - \Pi_{hp} v\|_{H^{-1/2}(\Gamma^2)}
= \sup_{w \in H^{1/2}(\Gamma^2)} 
\dfrac{\big<v - \Pi_{hp} v,w - \Pi_{hp} w \big>_{\Gamma^2}}
{\|w\|_{H^{1/2}(\Gamma^2)}}\\
&\leq 
\left(\sum_{J \in \meshBEM}
\dfrac{h_J}{p_J} \|v-\Pi_{hp} v\|_{L^2(J)}^2
 \right)^{1/2}  
\times \sup_{w \in H^{1/2}(\Gamma^2)}
\sum_{J \in \meshBEM} 
\left(\dfrac{p_J}{h_J} 
\dfrac{\|w-\Pi_{hp} w\|_{L^2(J)}^2}{\|w\|^2_{H^{1/2}(\Gamma^2)}}
\right)^{1/2}\\
&\leq 
C \left(\sum_{J \in \meshBEM} \dfrac{h_J^{2\rho+1}}{p_J^{2r+1}} 
\|v\|^2_{H^{r}(J)} \right)^{1/2},
\end{align*}
where, in order to bound the supremum, we used (\ref{estinf22}) and the estimate 
\begin{equation}
\sum_{J \in \meshBEM} \|w|_J\|_{H^{1/2}(J)}^2 \leq C \|w\|_{H^{1/2}(\Gamma^2)}^2
\end{equation}
from \cite{vPfDiss}, cf. also \cite[Lemma 3.4]{StSu91}.
\end{proof}
Based on the approximation results in Lemmas 
\ref{lem:estinf1} -- \ref{lem:estT45} and the quasi-optimality estimate from
Theorem \ref{thm:ba_bound} we derive an a priori error estimate in the 
quasi-uniform case. 

\begin{theorem}[Quasi-uniform \emph{a priori} error estimate]
\label{thm:apriori_sharp}
Suppose $u \in \fspace \cap [H^{m_1}(\Omega^1)\times H^{m_2}_{p.w.}(\Gamma^2)]$,
$m_1> 3/2$, $m_2>1$ is a weak solution of (\ref{p2}) and $U \in \spaceFBEM$ is a
solution of (\ref{weak_c}) or (\ref{weak_t}) with 
$\eta := (\delta+2)\kappa \invrate$ for arbitrary $\delta>0$. 
Moreover, assume that $f\in H^{-1}(\Omega)$ and 
$\p u^2/ \p\bn^2 \in H^{m_2-1}_{p.w.}(\Gamma^2)$. 
Suppose further that 
$\meshFEM$ is shape regular. Then, there holds
\begin{align} \label{apriori_sharp}
\|u - U\|_\norme{\kappa \invrate}^2
&\leq 
C \bigg\{ \gamma^1_{hp}
\|u^1\|_{H^{m_1}(\Omega^1)}^2 + 
\gamma^2_{hp}\|u^2\|_{H^{m_2}(\Gamma^2)}^2 
+\gamma^3_{hp}
\|\tfrac{\p u^2}{\p\mbf{n}^2}\|_{H^{m_2-1}_{p.w.}(\Gamma^2)}^2 
\bigg\} \nonumber
\end{align}
where for $K \in \meshFEM$, $J \in \meshBEM$,
\[
\gamma^1_{hp} := 
\left(
\max_{\overline{K} \cap \Gamma_I \neq \emptyset}
\dfrac{h_K^{\mu_1-1}}{p_K^{m_1-3/2}}
\right)^2,
\gamma^2_{hp} := 
\left(
\max_{J \subset \Gamma_I} 
\dfrac{h_J^{\mu_2-1/2}}{p_J^{m_2-1}} 
\right)^2,
\gamma^3_{hp} := 
\left( \max_{J \in \meshItwo} 
\dfrac{h_J^{\mu_2+1/2}}{p_J^{m_2+1/2}} \right)^2
\]
with $\mu_1 := \min\{m_1,p_K+1\}, \mu_2 := \min\{m_2,p_J+1\}$.
\end{theorem}

\begin{proof}
We are going to estimate the two terms in \eqref{ba_bound}, i.e. 
\mbox{$\inf_{\Phi \in \spaceFBEM}\|u - \Phi\|_\norme{\kappa \invrate}$} 
and the residual term 
$\res{\kappa \invrate}(u)$. 
For bounding the infimum we proceed by using the triangle inequality
for the jump term $\|[u - \Phi]\|_{1/2,\kappa\invrate}^2$ and the
continuity of $\hat \cS$ from \eqref{coel_hatS}. We obtain
\begin{align}
&\inf_{\Phi \in \spaceFBEM} \|u - \Phi\|_\norme{\kappa \invrate}^2\\
&~\leq \inf_{\Phi^1 \in \spaceFEM}
\left(\kappa_{\max} 
\|\nabla (u^1 - \Phi^1)\|_{L^2(\Omega^1)}^2
+ 2\sum_{J \in \meshIone} \kappa\invrate
\|u^1 - \Phi^1\|_{L^2(J)}^2\right) \nonumber \\[2ex]
&\qquad+ \inf_{\Phi^2 \in \spaceBEM}
\left(C_{\hat \cS} \|u^2 - \Phi^2\|_{H^{1/2}(\Gamma^2)}^2
+ 2\sum_{J \in \meshItwo} \kappa \invrate
\|u^2 - \Phi^2\|_{L^2(J)}^2\right) \nonumber\\[2ex]
&~\equiv T_\text{FE} + T_\text{BE}. \label{apinf12}
\end{align}
Recall the bound (\ref{invrate_est}) for $\invrate$. In order to estimate $T_\text{FE}$ we
choose $\Phi^1$ as $\Psi$ in Lemma \ref{lem:estinf1} and  by use of the interface indicator
function $\lambda_K$ for $K\in\meshFEMsig$, where
$\lambda_K=1$ for $\overline{K}\cap\Gamma_I\neq \emptyset$ and $\lambda_K=0$ otherwise,

we obtain for $m_1>3/2$
\begin{align}
T_\text{FE} &\leq C \left(\sum_{K \in \meshFEM} \left( \dfrac{h_K^{2\mu_1-2}}{p_K^{2m_1-2}} 
+ \lambda_K\invrate_K 
\dfrac{h_K^{2\mu_1-1}}{p_K^{2m_1-1}} \right)
\|u^1\|_{H^{m_1}(K)}^2\right)^{1/2} \nonumber\\
&\leq C \left(\sum_{K \in \meshFEM} \left( \dfrac{h_K^{2\mu_1-2}}{p_K^{2m_1-2}} 
+ \lambda_K
\dfrac{h_K^{2\mu_1-2}}{p_K^{2m_1-3}} \right)
\|u^1\|_{H^{m_1}(K)}^2\right)^{1/2} \nonumber\\
&\leq C 
\max_{\overline{K} \cap \Gamma_I \neq \emptyset}
\dfrac{h_K^{\mu_1-1}}{p_K^{m_1-3/2}}\cdot
\|u^1\|_{H^{m_1}(\Omega^1)} \label{ap1}
\end{align}
for $\mu_1 := \min\{m_1,p_K+1\}$.
Estimating $T_\text{BE}$, we 
choose $\Phi^2$ as $\Psi$ in Lemma \ref{lem:estinf2}, which yields for $m_2 > 1$
\begin{equation}\label{ap2}
T_\text{BE} \leq C 
\max_{J \subset \Gamma_I} 
\dfrac{h_J^{\mu_2-1/2}}{p_J^{m_2-1}} \cdot
\|u^2\|_{H^{m_2}(\Gamma^2)},
\end{equation}
where $\mu_2 := \min\{m_2,p_J+1\}$. 

We proceed with bounding the residual 
$\res{\kappa \invrate}(u)$ by estimating the two terms 
in (\ref{res_bound}), which we denote by $R_\textrm{FE}$ and 
$R_\textrm{BE}$, respectively.
For $J \in \meshIone$ denote by $K_J \in\meshFEM\subset \Gstrip$ the element
for that holds $J \subset \p K_J$. 
Componentwise application of Lemma \ref{lem:estT3} 
yields for $\nabla u^1 \in H^{m_1-1}(K_J),$ since $m_1-1>1/2,$
with the definition of $\|\cdot\|_{-1/2,\kappa\invrate}$ 
(cf. \eqref{def-1/2norm}) that
\begin{equation}\label{ap3}
\begin{split}
R_\textrm{FE} = 
\|\kappa (\nabla u^1 - \bPi_{hp}(\nabla u^1))\|_\normm{\kappa \invrate}
&\leq 
C \left( \sum_{J \in \meshIone} 
\dfrac{1}{\invrate_{K_J}}
\dfrac{h_K^{2\mu_1-3}}{p_K^{2m_1-3}}
\|u^1\|_{H^{m_1}(K_J)}^2
\right)^{1/2}\\
&\leq 
C 
\max_{\overline{K} \cap \Gamma_I \neq \emptyset}
\dfrac{h_K^{\mu_1-1}}{p_K^{m_1-1/2}} \cdot
\|u^1\|_{H^{m_1}(\Omega^1)},
\end{split}
\end{equation}
where $\mu_1$ is as above.

Furthermore, using Lemma \ref{lem:stekcons}, 
we find
\begin{equation*}
R_\textrm{BE} = 
\|(\cS - \hat \cS)u^2 - (\cN - \hat\cN)f^2\|_{H^{-1/2}(\Gamma^2)}
\leq C \inf_{\Phi^2 \in \spaceBEMtrac}\|
\tfrac{\p u^2}{\p\bn^2} - \Phi^2\|_{H^{-1/2}(\Gamma^2)}.
\end{equation*}
Choosing $\Phi^2$ to be the elementwise $L^2$-projection of 
$\p u^2/ \p\bn^2$ in combination with Lemma \ref{lem:estT45} 
we obtain the optimal bound
\begin{align}\label{ap4}
T^4 &\leq C \max_{J \in \meshBEM}{\dfrac{h_J^{\mu_2+1/2}}{p_J^{m_2+1/2}}}
\left\|
\dfrac{\p u^2}{\p\bn^2}
\right\|_{H^{m_2-1}_{p.w.}(\Gamma^2)}.
\end{align}
Collecting the estimates \eqref{ap1},\eqref{ap2},\eqref{ap3},
and \eqref{ap4} proves the claim.
\end{proof}

\subsection{$hp$-Nitsche with corner singularities}

For the $hp$-analysis we cite approximation properties of the Finite Element and Boundary
Element spaces, respectively. The given construction from Section \ref{subsec:Discr} is
compatible with the results we will use to prove exponential convergence of the $hp$-Nitsche
method. It is well known that the problem \eqref{p1} has a unique solution and that when the
solution $u$ of \eqref{p1} belongs to the class $\cB^2_{\beta}(\Omega)$ then the $hp$-Version
of the Finite Element method converges exponentially, cf. e.g. \cite[Theorem 3.1]{BabGuoPart2}
and the references therein. 

\begin{theorem}[\cite{BabGuoPart2}]\label{thm:exFEsol}
Let $\Omega\subset\mathbb{R}^2$ be a straight edge polygon. Then there exist $\beta\in(0,1)$,
such that for given data
$f\in\mathcal{B}^0_\beta(\Omega)$, $g_D\in \mathcal{B}^{3/2}_\beta(\Gamma_D)$ and
$g_N\in\mathcal{B}^{1/2}_\beta(\Gamma_N)$ the weak solution 
of corresponding variational formulation of \eqref{p1} $u\in H^{1}(\Omega)$ exists and belongs
to $\mathcal{B}^2_\beta(\Omega)$.
\end{theorem}

\begin{remark} By Theorem \ref{thm:regtrace} $g_D\in\cB^{3/2}_\beta(\Gamma_D)$ can be
interpreted as $g_D\in\cB^{1,2}_{\hat{\beta}}(\Gamma_D)$ for certain $\hat{\beta}$.
Similarly, $g_N\in\cB^{1/2}_\beta(\Gamma_N)$ implies $g_N\in\cB^{0,1}_\beta(\Gamma_N)$.
Moreover, cf. \cite[Remark 3.2]{BabGuoSteph90}, we have for 
$g\in\cB^{0,1}_{\hat\beta}(\Gamma^2)$ that $g\in H^{-1/2}(\Gamma^2)$. In particular, for
$u\in\cB^{2}_\beta(\Omega)$ we have $\p u/\p\bn|_{\Gamma}\in\cB^{0,1}_{\hat{\beta}}(\Omega)$,
where $\hat{\beta}$ is given as in Theorem \ref{thm:regtrace}.
\end{remark}

\begin{lemma}[{\cite[Lemma 4.16,4.25]{Swb98}}]\label{approxK0}
For $K\in\frak{S}_0$ and $u\in H^{2,2}_\beta(K)$ the (bi)linear interpolant 
$\Psi$ of $u$ in the verticess of $K_0$ satisfies
\begin{equation*}
\begin{split}
	\|u-\Psi\|^2_{H^1(K)} 
	\leq& C h^{2(1-\beta)}_{K} |u|^2_{H^{2,2}_\beta(K)}
	\leq C \sigma^{2(n+1)(1-\beta)} \| \Phi^1_\beta\:|D^2 u| \|^2_{L^2(K)}.
\end{split}
\end{equation*}
\end{lemma}

\begin{lemma}[{\cite[Lemma 4.53]{Swb98}}]\label{hpapprox1}
For any $u\in H^{k+3,2}_\beta(K)$ and every $K\in \meshFEMsig$ there exists 
$\Psi\in\spaceFEMsig$ of degree $p_K$ in each variable, such that
\[
\|D^m (u-\Psi)\|^2_{L^2(K)} \leq C \sigma^{2 (n + 2 - j) (2 - m - \beta)}
\frac{\Gamma(p_K - s_K + 1)}{\Gamma(p_K + s_K + 3 - 2m)} 
\left( \rho/2 \right)^{2s_K} \|u\|^2_{H^{s_K+3,2}_\beta(K)}
\]
for $0\leq m \leq 2$ and $1\leq j \leq n+1, K\in \frak{S}_j$ and any 
$1\leq s_K \leq \min(p_K,k)$, where $\rho = \max\{1,(1-\sigma)/\sigma\}$.
\end{lemma}
\noindent A proof for the following result may be found in \cite{Quart84}.
\begin{lemma}[Inverse inequality]\label{lem:invineq} Let $I=(a,b)$ be a bounded interval and 
$h = b-a$. Then, for
every $\phi\in\cP_p(I)$ it holds that
\[
	\|\phi\|_{L^\infty(I)} \leq 2 \left(\frac{8}{h}\right)^{1/q} p^{2/q} 
	\|\phi\|_{L^q(I)},\qquad 1\leq q\leq \infty.
\]
\end{lemma}

\begin{theorem}[\emph{A priori} $hp$-estimate]\label{thm:apriori_exp} 
Let $f\in\cB^0_\beta(\Omega)$.
Suppose $u=(u^1,u^2) \in V\cap 
[\mathcal{B}^2_{\beta}(\Omega^1)\times \mathcal{B}^{1,2}_{\hat\beta}(\Gamma^2)]$
with $0\leq\beta,\hat\beta<1$ is the weak solution of (\ref{p2}) and 
$U \in \spaceFBEMsig$ is a solution of (\ref{weak_c}) or (\ref{weak_t}) 
with $\eta := \eta_0\kappa \invrate$ for an arbitrary but large enough constant
$\eta_0>0$ and any $\sigma\in(0,1)$ for each subdomain. 
Then there exists a slope $\mu>0$ for each subdomain, such that there holds
\begin{align}
\|u - U\|_\norme{\eta}^2 \leq C e^{-2 b p},
\end{align}
where the constants $C,b$ are positive and do not depend on the number of layers
$n+1$ of the geometric meshes and $p$ grows proportionally with the layers. 
\end{theorem}
\begin{proof} 
For the sake of simplicity (compare also the discussion in e.g.
\cite{GuoHeuSte96} or \cite{BabGuoSteph90}) 
we shall consider the L-shaped domain $\Omega=[-1,1]^2 \setminus [0,1]^2$, 
where only one singularity is situated at the reentrant corner at the origin 
and the interface $\Gamma_I$ is a single straight arc from $(-1,-1)^\top$ to
$(0,0)^\top$ (cf. Figure \ref{figSingB}). 
The general case of multiple singularities and 
an interface consisting of more than one straight arc follows similarly.
Moreover, let $\overline\Gamma^2 = \cup_{i=1}^L \overline\Gamma^2_i$ be a
collection of straight arcs with $\Gamma_I = \Gamma^2_{L}$, and 
$\Gamma^2_1$ be the arc from $(0,0)^\top$ to $(0,1)^\top$.

Note that for all straight arcs $\gamma$ that
divide $\Omega$ into two polygonal domains $\Omega^1$ and $\Omega^2$
as well as for parts of the boundary $\gamma \subset \p \Omega$ 
for $u\in B^2_\beta(\Omega)$, there holds 
$u|_{\gamma}\in B^{1,2}_{\hat\beta}(\gamma)$ by Theorem \ref{thm:regtrace}. 
Let us assume that $u^2\in B^1_{\hat\beta_k}(\Gamma_k)$ for $k\in\{1,L\}$
with $0 < \hat\beta_k < 1/2$ and $u^2$ analytic on $\Gamma_\ell$ for 
$1 < \ell < L$. The case 
$u^2\in B^2_{\hat\beta_k}(\Gamma_k)$ for $k\in\{1,L\}$ follows similarly, cf. 
\cite[Remark after Thm. 5.5]{GuoHeuSte96}.

We proceed with the analysis in four steps and analogously 
to Theorem \ref{thm:apriori_sharp}, i.e. first we analyse the FE part, then the
BE part, and deal with terms due to the residual in a third step. Then we
conclude in step four.
In the following, we write $\lesssim$ to denote $\leq C$ with a generic constant $C$.
Using the same notation as in \eqref{apinf12}, 
we start from the quasi-optimality result from \eqref{ba_bound}:
\begin{equation*}
\|u-U\|^2_{\norme{\eta}} \lesssim
 \inf_{\Phi \in \spaceFBEMsig} \|u - \Phi\|_\norme{\eta}^2 
+ \res{\eta}(u)^2 
\equiv T_\textrm{FE} + T_\textrm{BE} + R.
\end{equation*}
\textbf{Step 1:} 
Using a scaled variant of the trace inequality, cf. e.g. 
\cite[(3.23)]{HouSwbSueli02},
\begin{equation}\label{eq:multtracescale}
	\|v\|^2_{L^2(\partial K)} \lesssim \left(h_K\|\nabla v\|^2_{L^2(K)} 
	+ h^{-1}_K \|v\|^2_{L^2(K)} \right).
\end{equation}
we find
\begin{equation*}
\begin{split}
T_\textrm{FE} 
&\leq C \kappa_{\max} \inf_{\Phi \in \spaceFEMsig}
\left\{
	\sum\limits_{K\in\meshFEMsig} \|\nabla (u^1-\Phi^1)\|^2_{L^2(\Omega^1)}
	\right.\\
   &\left.\qquad
   \qquad\quad
   +\sum\limits_{J\in\cT^\sigma_{hp}(\Gamma^1_I)} \frac{p^2_{K_J}}{h_{K_J}}
   		\left( h_{K_J} \|\nabla(u^1 - \Phi^1)\|^2_{L^2({K_J})} 
   			  +h_{K_J}^{-1} \|u^1 - \Phi^1\|^2_{L^2({K_J})}
   	    \right) 
\right\}\\
&\leq C \kappa_{\max} \inf_{\Phi \in \spaceFEMsig}
	\sum\limits_{K\in\meshFEMsig} 
	p^2_K 
	\left(
    \|\nabla (u^1-\Phi^1)\|^2_{L^2(K)}
   +h_{K_J}^{-2} \|u^1 - \Phi^1\|^2_{L^2(K)}
   	    \right), \\
\end{split}   	    
\end{equation*}
where we have used the bound \eqref{invrate_est} for $\invrate$. Then we
construct a globally continuous approximant $\Psi^1$ as in 
\cite[Lemma 4.57]{Swb98} and further obtain
\begin{equation}
\begin{split}
T_\text{FE}\lesssim& \quad\quad~~ \sum\limits_{K\in\frS_0}
h^{2(1-\beta)} \|u^1-\Psi^1\|^2_{H^{2,2}_\beta(K)}\\
&\quad+\sum\limits_{\overset{K\in\frS_j:}{1\leq j \leq n+1}} 
p^2_K \left(  h^{-2}_K\|u^1-\Psi^1\|^2_{L^2(K)}
+ \|u^1-\Psi^1\|^2_{H^1(K)}
+ h^{2}_K\|u^1-\Psi^1\|^2_{H^2(K)}\right)\\
&\quad+~~\:\sum\limits_{K\in\frR}  
p^2_K\left(  h^{-2}_K\|u^1-\Psi^1\|^2_{L^2(K)}
+ \|u^1-\Psi^1\|^2_{H^1(K)}
+  h^{2}_K\|u^1-\Psi^1\|^2_{H^2(K)}\right).
\end{split}
\end{equation}
Using Lemma \ref{approxK0} and Lemma \ref{hpapprox1} as well as the estimates
for $|\cdot|_{H^{m,\ell}_\beta(K)}$ from \eqref{defBlOmega} and \cite[(4.5.4)]{Swb98}, i.e. for
$u\in B^\ell_\beta(\Omega^1)$ for any $k\geq\ell$ there holds 
\[
\|u^1\|_{H^{k+\theta,\ell}_\beta(\Omega^1)} \leq 
C d^{k+\theta-\ell}
(k+\theta)^{1/2}
\Gamma(k+\theta-\ell+1),
\]
for $1\leq s_K \leq \min(p_K,k)$ yields the bound
\begin{equation*}
\begin{split}
T_\textrm{FE} 
\lesssim& 
	\left( \sum\limits_{K\in\frS_0}
	 \sigma^{2(n+1)(1-\beta)} \|u^1\|^2_{H^{2,2}_{\beta}(K)} \right.\\
	 &
	 \qquad+\:  \sum\limits_{\overset{\normalsize K\in\frS_j:}{1\leq j \leq n+1}}  
	     p^2_K \sigma^{2 (n + 2 - j) (1 - \beta)}
		 \frac{\Gamma(p_K - s_K + 1)}{\Gamma(p_K + s_K - 1 )} 
		 \left( \rho/2 \right)^{2s_K} \|u^1\|^2_{H^{s_K+3,2}_\beta(K)}\\
	 &\left.
	 \qquad+\quad	 
		\sum\limits_{K\in\frR}  
	     p^2_K
		 \frac{\Gamma(p_K - s_K + 1)}{\Gamma(p_K + s_K - 1 )} 
		 \left( \frac{h_K}{2} \right)^{2s_K} \|u^1\|^2_{H^{s_K+3}(K)}
		 \right)\\
\lesssim& 
	\sigma^{2(n+1)(1-\beta)}\\
	&\:\:\times\left( \sum\limits_{K\in\frS_0}
	 1 +
	   \sum\limits_{\overset{K\in\frS_j:}{1\leq j \leq n+1}}  
	     s_K p^2_K \sigma^{2 (1 - j) (1 - \beta)}
   		 \left( \frac{\rho d}{2} \right)^{2s_K} 
		 \frac{\Gamma(p_K - s_K + 1)}{\Gamma(p_K + s_K - 1 )} 
		  \Gamma(s_K+2)^2
		 \right)		 
\end{split}
\end{equation*}
and, since $u^1$ is analytic on $K\in\frR$, increasing the polynomial degree on
each element $K\in\frR$ in accordance with the polynomial degree in $\frS_{n+1}$
yields exponential convergence for the sum regarding the elements $K\in\frR$,
cf. \cite[Section 4.5.3]{Swb98}. We are left to bound the sum involving the
layers $\frS_j,j=0,\dots,n+1$.

For $K\in\frS_j$ we choose $s_K \equiv s_j = \alpha_j p_j$ with 
$\alpha_j\in(0,1)$. Then using Stirling's formula, i.e. 
$\Gamma(n+1) \sim \sqrt{2\pi n}(\tfrac{n}{e})^n$, it is easy to show that 
\[
\left( \frac{\rho d}{2} \right)^{2s_j} 
\frac{\Gamma(p_j - s_j + 1)}{\Gamma(p_j + s_j + 1 )} 
\Gamma(s_j+1)^2 \leq C p_j (F(\rho d,\alpha_j))^{p_j},\quad
F(d,\alpha):= \left(\frac{\alpha d}{2}\right)^{2\alpha} 
\frac{(1-\alpha)^{(1-\alpha)}}{(1+\alpha)^{(1+\alpha)}},
\]
which further yields
\begin{equation}\label{eqhpbracket}
\begin{split}
T_\text{FE}
\lesssim& 
	\sigma^{2(n+1)(1-\beta)}
	 \left( 1+ 
	   \sum\limits_{1\leq j\leq n+1}  
	     \sigma^{2 (1 - j) (1 - \beta)}
	     \alpha^3_j p^6_j (F(\rho d,\alpha_j))^{p_j}
		 \right).	 
\end{split}
\end{equation}
Now choosing $\alpha_j=\max\{1/p_j,\alpha_{\min}\},j\geq 1$, where 
$\alpha_{\min}$ is deduced from 
\[F(d,\alpha_{\min})=\inf_{\alpha\in(0,1)} F(d,\alpha).\]
Letting $F_{\min} := F(\rho d,\alpha_{\min})$ we find a lower bound for 
the slope 
$\mu=\mu(\sigma,\beta, F_{\min})$ analogously to 
\cite[Theorem 3.36, Theorem 4.51]{Swb98} and thus a suitable $\mu$ exists.
Selecting $p_j = \max \{2, \lfloor  \mu j\rfloor\}$, 
it can be shown (loc. cit.) that the parenthesis in \eqref{eqhpbracket} is
bounded independently of $n$ and we arrive at
\begin{equation}\label{expapFE}
T_\text{FE} \lesssim \sigma^{2(n+1)(1-\beta)}
\leq C e^{-2b_\text{FE} p_\text{FE}},
\end{equation}
where $C>0$ and $b_\textrm{FE}= -(1-\beta)\ln(\sigma) > 0$
do not depend on the discretisation parameters and $p_\text{FE}$
grows proportionally with the number of layers in the geometric finite
element mesh.\\

\noindent\textbf{Step 2:}
In order to estimate $T_\textrm{BE}$, we choose $\Psi^2\in\spaceBEMsig$ as in 
\cite[Theorem 5.5]{GuoHeuSte96}, such that
\begin{equation*}
\|u^2-\Psi^2\|^2_{H^{1/2}(\Gamma^2)} \leq C e^{-2b_{\text{BE},1} p_\text{BE}}
\end{equation*}
follows immediately with $C>0, b_\text{BE}>0$, and $p_\text{BE}$
growing proportinally with the numbers of layers in the geometric boundary
element mesh. We are left to control the jump term on the boundary element side.
We proceed as in the proof of \cite[Theorem 5.5]{GuoHeuSte96}, also compare 
\cite[Chapter 3]{Swb98}. 

Using \cite[Lemmas 5.3--5.4]{GuoHeuSte96} (cf. also 
\cite[Lemma 3.39,Lemma 3.41]{Swb98}),which are one
dimensional counterparts of Lemma \ref{approxK0} and Lemma \ref{hpapprox1},
since $u\in B^{1}_{\hat\beta_L}(\Gamma^2_L)$, with 
$1\leq s_{J_i}\leq p_i,i\geq 1$ we arrive at
\begin{equation}\label{eq:expap2}
\begin{split}
&\sum\limits_{J\in\cT^{\sigma}_{hp}(\Gamma^2_I)} 
	\kappa\invrate \|u^2-\Psi^2\|^2_{L^2(J)}\\
&\quad\leq C \kappa_{\max} 
	\Bigg(
	 h_{J_0}^{-1} \sigma^{2(n+1)(1-\hat\beta_L)} 
	 |u^2|^2_{H^{2, 1}_{\hat\beta_L}(J_0)} 
	 \\
	 &\qquad+  \sum\limits_{i=1}^{n}  
	     p^2_{J_i} h^{-1}_{J_i}
	     \sigma^{2(n + 2 - i) (1 - \hat\beta_L)}
		 \frac{\Gamma(p_{J_i} - s_{J_i} + 1)}{\Gamma(p_{J_i} + s_{J_i} + 3)} 
		 \left( \frac{\rho}{2} \right)^{2s_{J_i}} 
		 \|u^2\|^2_{H^{s_{J_i}+1}_{\hat\beta_L}(J_i)}
	     \Bigg)\\
&\quad\leq C \kappa_{\max} 
	\sigma^{2(n+1)(1/2-\hat\beta_L)} \\
	&\qquad\times
	\Bigg(
	 1 + 
	   \sum\limits_{i=1}^{n}  
	     p^2_{J_i} \sigma^{2(1 - i) (1/2 - \hat\beta_L)}
		 \frac{
		 \Gamma(p_{J_i} - s_{J_i} + 1)
		 }{\Gamma(p_{J_i} + s_{J_i} + 3 )} 
		 \left(\frac{\rho d}{2} \right)^{2s_{J_i}}\Gamma(s_{J_i}+1)^2 		 
	\Bigg),
\end{split}
\end{equation}
where $J_i$ realizes an enumeration of the layers of
the geometric mesh on $\Gamma_{L}$.
Now, following the lines of the proof of \cite[Thm. 5.5]{GuoHeuSte96}, we can
bound the parenthesis in \eqref{eq:expap2} similar to \eqref{eqhpbracket},
which yields
\begin{equation*}	
\sum\limits_{J\in\cT^{\sigma}_{hp}(\Gamma^2_I)} 
	\kappa\invrate \|u^2-\Psi^2\|^2_{L^2(J)}
	\leq C \kappa_{\max} \sigma^{2(n+1)(1/2-\hat\beta_L)}
	\leq C e^{-2b_{\text{BE},2} p_\text{BE}}.
\end{equation*}
Collecting 
the terms shows that
\begin{equation}\label{expapBE}
	T_\textrm{BE} \leq C e^{-2 b_{\text{BE},1} p_\text{BE}} 
	+ C e^{-2 b_{\text{BE},2} p_\text{BE}}
	\leq C e^{-2 b_\text{BE} p_\text{BE}},	
\end{equation}
with $b_\textrm{BE} = \min(b_\textrm{BE,1},b_\textrm{BE,2})$.\\
\textbf{Step 3:}
We proceed with bounding the residual $\res{\kappa \invrate}(u)$. 
By \eqref{eq:resrep}
\[
	\resphi(u,\Phi) = 
	\sum\limits_{J\in\cT^\sigma_{hp}(\Gamma^1_I)}
	\big< \kappa (\nabla u^1 - \bPi_{hp}( \nabla u^1))\cdot \bn^1,
		 [\Phi]\big>_{J}
  - \big<(\cS - \hat \cS) u^2 - (\cN - \hat \cN) f^2,
	      \Phi^2 \big>_{\Gamma^2}.
\]
We note that for $J_0\subset\p K_0$ with $K_0\in\frS_0$ 
the components of $\nabla u^1$ belong to $H^{1,1}_\beta(K_0)$ and 
there holds $\nabla u^1\in L^1(J_0)$. We find by virtue of Lemma \ref{lem:invineq}
\[
\begin{split}
\big< \kappa (\nabla u^1 - \bPi_{hp}( \nabla u^1))\cdot\bn^1,[\Phi]\big>_{J_0}
\lesssim&\:  \kappa_{\max} 
\|\nabla u^1 - \bPi_{hp}(\nabla u^1)\|_{L^1(J_0)}
\|[\Phi]\|_{L^\infty(J_0)} \\
\lesssim&\: 
\kappa_{\max}  
\|\nabla u^1 - \bPi_{hp}(\nabla u^1)\|_{L^1(J_0)}
\tfrac{p_{J_0}}{h^{1/2}_{J_0}} \|[\Phi]\|_{L^2(J_0)}\\
\lesssim&\:  \kappa_{\max}  
\|\nabla u^1 - \bPi_{hp}(\nabla u^1)\|_{L^1(J_0)}
\|\Phi\|_{\norme{\eta}},
\end{split}	
\]
while for all $J \not = J_0$  by \eqref{eq:weightedCS} there holds
\[
\begin{split}
\big< \kappa (\nabla u^1 - \bPi_{hp}( \nabla u^1))\cdot \bn^1,
[\Phi]\big>_{J} 
\leq&\: 
\|\eta^{-1/2}(\nabla u^1 - \bPi_{hp}( \nabla u^1))\|_{L^2(J)} 
\|\eta^{1/2}[\Phi]\|_{L^2(J)} \\
\leq& \: 
\|\eta^{-1/2}(\nabla u^1 - \bPi_{hp}( \nabla u^1))\|_{L^2(J)}
\|\Phi\|_{\norme{\eta}}.
\end{split}
\]
For the consistency error from the boundary element part we infer
\[
\begin{split}
	\big<(\cS - \hat \cS) u^2 - (\cN - \hat \cN) f^2,\Phi^2 \big>_{\Gamma^2}
	\leq& \|(\cS - \hat \cS) u^2 - (\cN - \hat \cN) f^2\|_{H^{-1/2}(\Gamma^2)} 
	\|\Phi^2\|_{H^{1/2}(\Gamma^2)}\\
	\leq& c^{-1/2}_{\hat\cS}
	\|(\cS - \hat \cS) u^2 - (\cN - \hat \cN) f^2\|_{H^{-1/2}(\Gamma^2)} 
	\|\Phi\|_{\norme{\eta}}.
\end{split}	
\]
Hence, with an enumeration of the edges $J\in\cT^\sigma_{hp}(\Gamma^1_I)$
\[
\begin{split}
	\big(\res{\eta}(u)\big)^2 \lesssim&\:
	\|\nabla u^1 - \bPi_{hp}(\nabla u^1)\|^2_{L^1(J_0)} +
	\sum\limits_{\ell=1}^n 
	\|\eta^{-1/2}(\nabla u^1 - \bPi_{hp}( \nabla u^1))\|^2_{L^2(J_\ell)} \\
	&+\|(\cS - \hat \cS) u^2 - (\cN - \hat \cN) f^2\|^2_{H^{-1/2}(\Gamma^2)} 
	\equiv R_\textrm{FE} + R_\textrm{BE}.
\end{split}	
\]
For $R_{\textrm{FE}}$, by setting $\xi:=\nabla u^1 - \bPi_{hp}(\nabla u^1)$,
we start by bounding the $L^1$-trace on $J_0$ using a standard trace theorem
with a scaling argument 
\[
	\|\xi\|_{L^1(J_0)} \leq C \left(h_{K_{J_0}}^{-1} \|\xi\|_{L^1(K_{J_0})}
	+ \|\nabla\xi\|_{L^1(K_{J_0})}
	\right).
\]
Then, since $p_{K_{J_0}}=1$ by construction, using \cite[Lemma 11.18]{ErnGuerVolI}
componentwise for $m=0,1$ and $r=1$ we arrive at
\[
	\|\nabla u^1 - \bPi_{hp}(\nabla u^1)\|_{L^1(J_0)} 
	\leq C |\nabla u^1|_{W^{1,1}(K_{J_0})}.
\]
Since $D^2 u^1\in H^{0,0}_\beta(K_{J_0})$ and 
$H^{0,0}_\beta(K_{J_0}) \subset L^1(K_{J_0})$, we can further estimate the right hand side 
(cf. \cite[Lemma 1.3.2]{Wihler02}) and obtain
\begin{equation}\label{eq:L1tracebnd}
	\|\nabla u^1 - \bPi_{hp}(\nabla u^1)\|_{L^1(J_0)} 
	\leq C h^{1-\beta}_{K_{J_0}} |u^1|_{H^{2,2}_\beta(K_{J_0})}.
\end{equation}
For the remaining terms of $R_\textrm{FE}$ we employ 
\cite[Lemma 3.9]{HouSwbSueli02}, which yields for any $u\in H^k(K), k\geq 1, p\geq 0$ and
integer $s\in\{1:\min(p+1,k)\}$
\begin{equation}\label{eq:L2projtracebnd}
	\|u-\Pi_{hp}\|_{L^2(\partial K)} \leq C \Phi_1(s,p) h^{s-1/2}_K 
	|u|_{H^s(K)},
\end{equation}
where 
\begin{equation*}
\begin{split}
	\Phi_1(s,p) := (2p+1)^{-1/2} &\left[ 
	\left( \frac{\Gamma(p+2-s)}{\Gamma(p+s)} \right)^{1/2}+
	\left( \frac{\Gamma(p+3-s)}{\Gamma(p+1+s)} \right)^{1/2}
	\right]\\ 
	&+ 
	\left( \frac{\Gamma(p+2-s)}{\Gamma(p+2+s)}\cdot 
		   \frac{\Gamma(p+3-s)}{\Gamma(p+1+s)}
	\right)^{1/4}
	+ \left( \frac{\Gamma(p+2-s)}{\Gamma(p+2+s)} \right)^{1/2}.
\end{split}
\end{equation*}
We can simplify the expression for $\Phi_1(s,p)$ --- by using the inequality 
$(a+b)^2\leq 2a^2+2b^2$ twice, the identity $\Gamma(z+1)=z\Gamma(z)$ and 
pulling out the factor $\tfrac{\Gamma(p+2-s)}{\Gamma(p+s)}$ --- to
\begin{equation*}
\Phi_1(s,p)^2 \leq 16 \cdot \frac{\Gamma(p+2-s)}{\Gamma(p+s+1)}.
\end{equation*}

Hence, we obtain by employing \eqref{eq:L1tracebnd} and 
\eqref{eq:L2projtracebnd} for $i\geq 1$ with $s_i\in\{1:\min(p_{K_{J_i}}+1,k_i)\}$ for any 
$u^1\in H^{k_i}(K_{J_i})$ with integer $k_i\geq 1$ the following bound
\begin{equation}\label{expbndRFE}
\begin{split}
 &
 	\|\nabla u^1 - \bPi_{hp}(\nabla u^1)\|^2_{L^1(J_0)} +
	 \sum\limits_{i=1}^n 
	\|\eta^{-1/2}(\nabla u^1 - \bPi_{hp}( \nabla u^1))\|^2_{L^2(J_i)} \\
	\lesssim& 
	h^{2(1-\beta)}_{K_{J_0}} |u^1|^2_{H^{2,2}_\beta(K_{J_0})}
	+
	\sum\limits_{i=1}^n 
	\frac{h_{K_{J_i}}}{p^2_{K_{J_i}}}
	h^{2s_i-1}_{K_{J_i}}  \frac{\Gamma(p_{K_{J_i}}-s_i+2)}{\Gamma(p_{K_{J_i}}+s_i+1)} 
	|u|^2_{H^{s_i}(K_{J_i})} \\   
	\lesssim& 
	\sigma^{2(n+1)(1-\beta)} |u^1|^2_{H^{2,2}_\beta(K_{J_0})}\\
	&\qquad\quad+ \sum\limits_{i=1}^n 
	\sigma^{2(n+2-i)(1-\beta)}  
	\frac{\Gamma(p_{K_{J_i}}-s_i+1)}{\Gamma(p_{K_{J_i}}+s_i-1)} 
	\left(\frac{\rho}{2}\right)^{2s_i}
	\|u\|^2_{H^{s_i+3,2}_\beta(K_{J_i})},
\end{split}	
\end{equation}
where the last bound follows from the fact that $u\in\cB^2_\beta(\Omega^1)$, such that
$u$ is analytic on elements $K\not\in \frS_0$, and there holds 
\[
|u|_{H^{s_i}(K)} \leq |u|_{H^{s_i+3}(K)} 
\leq \|\Phi^1_{-(\beta+s_i+3-2)}\|_{L^2(K)} \|u\|_{H^{s_i+3,2}_\beta(K)},
\]
while keeping in mind that $h_{K_{J_i}}\leq \rho\text{dist}(0,K_{J_i})$ with 
$\rho = \max\{1,\tfrac{1-\sigma}{\sigma}\}$ for $i\geq 1$ 
(compare Lemma \ref{hpapprox1} and e.g. \cite[Lemma 4.48,4.50]{Swb98}).

Analogously to bounding the term $T_\textrm{FE}$, it can be shown that last bound of
\eqref{expbndRFE} is bounded independently of $n$, and thus
\begin{equation}\label{expRFE}
R_\textrm{FE} \leq C \sigma^{2(n+1)(1-\beta)}\leq C e^{-2 b_{R_\textrm{FE}}
p_\text{FE}},
\end{equation}
where $C>0, b_{R_\textrm{FE}}>0$ are independent of $n$ and $p_\text{FE}$ grows proportionally
with the number of layers in the geometric mesh of $\Omega^1$.

Finally, in a similar fashion as the quasi-uniform case we find that
\begin{equation*}
R_\textrm{BE} = 
\|(\cS - \hat \cS)u^2 - (\cN - \hat\cN)f^2\|^2_{H^{-1/2}(\Gamma^2)}
\leq C \inf_{\Phi^2 \in \spaceBEMgradsig}
\left\|\frac{\p u^2}{\p\bn^2} - \Phi^2\right\|^2_{H^{-1/2}(\Gamma^2)}.
\end{equation*}
Since $u^2\in\cB^{1}_{\hat\beta}(\Gamma^2)$, there holds
$\p u^2/\p\bn^2 \in \mathcal B^{0}_{\tilde{\beta}}(\Gamma^2)$ for 
suitable $\tilde{\beta}$. Letting 
$z(s) = \int^s_0 \tfrac{\p u^2}{\p\bn^2}(t)\textrm{d} t$,
where integration is carried out along $\Gamma^2$ and $s$ is the arc length parameter of 
$\Gamma^2$, then 
$z\in\cB^1_{\tilde{\beta}}(\Gamma^2)$. Since, as above, by 
\cite[Theorem 5.5]{GuoHeuSte96} there exists a polynomial $\psi$, such that 
\[
\|z-\psi\|_{H^{1/2}(\Gamma^2)} \leq C e^{-b p_\textrm{BE}}.
\]
Set $\chi = \psi'$. Then due to \cite[Lemma 3.2]{SteSur89}, there holds
$(z-\psi)'=\tfrac{\p u^2}{\p\bn^2} - \chi\in H^{-1/2}(\Gamma^2)$ and
\[
	\|\tfrac{\p u^2}{\p\bn^2} - \chi\|_{H^{-1/2}(\Gamma^2)}
	\leq C \|z - \psi\|_{H^{1/2}(\Gamma^2)},
\]
which yields
\begin{equation}\label{expRBE}
 \|(\cS - \hat \cS)u^2 - (\cN - \hat\cN)f^2\|^2_{H^{-1/2}(\Gamma^2)}
 \leq C e^{-2b_{R_\textrm{BE}} p_{\text{BE}}}.
\end{equation}

\textbf{Step 4:}
Collecting the estimates \eqref{expapFE},\eqref{expapBE},\eqref{expRFE} and
\eqref{expRBE} yields the assertion with 
$b=\min\{b_\text{FE},b_{\text{BE}},b_{R_1},b_{R_2}\}$ and 
$p = \min\{ p_\text{BE}, p_\text{FE}\}$.
\end{proof}

\section{Algebraic formulation}

The algebraic system corresponding to the weak formulation (\ref{weak_c}) will be described in
this section. We denote by $U^1_I$ and $U^2_I$ the degrees of freedom associated with the
interface $\Gamma_I$ from the FE side and from the BE side, respectively. The remaining degrees
of freedom from the FE side are denoted by $U^1_O$ and from the BE side by $U^2_O$. Then the
algebraic problem has the following structure

\begin{equation*}
(\frA + \frB + \frC)
\left( 
\begin{array}{c}
U^1_O\\
U^1_I\\
U^2_I\\
U^2_O
\end{array}
\right)=
\left( 
\begin{array}{c}
l^1_O\\
l^1_I\\
l^2_I\\
l^2_O
\end{array}
\right).
\end{equation*}

The matrix $\frA$ is the stiffness matrix of the finite element and the boundary element part
produced with the term 
$(\kappa \nabla U, \nabla \Phi)_{\Omega^1} + \big<\hat \cS U, \Phi \big>_{\Gamma^2}$ in
(\ref{weak_c}) without coupling terms,

\begin{equation*}
\frA:=
\left( 
\begin{array}{cccc}
A_{OO}&A_{IO}^T&0&0\\
A_{IO}&A_{II} &0&0\\
0&0  &S_{II}&S_{IO}^T\\
0&0  &S_{IO}&S_{OO}\
\end{array}
\right)
:=
\left( 
\begin{array}{cc}
\cA&0\\
0&\cS
\end{array}
\right).
\end{equation*}
Terms $- \big< \kappa \nabla U^1 \cdot \bn^1,[\Phi] \big>_{\Gamma_I}
- \big<[U],\kappa \nabla \Phi^1 \cdot \bn^1 \big>_{\Gamma_I}$ in (\ref{weak_c}) constitute the
matrix $\frB$,
\begin{equation*}
\frB:=
\left( 
\begin{array}{cccc}
0&0&(B_{IO}^{12})^T&0\\
0& B_{II}^{11}+(B_{II}^{11})^T &(B_{II}^{12})^T&0\\
B_{IO}^{12}&B_{II}^{12}& 0 &0\\
0&0&0&0
\end{array}
\right).
\end{equation*}
Finally, the stabilization term $\big< \eta[U], [\Phi]  \big>_{\Gamma_I}$ in (\ref{weak_c})
constitutes the matrix $\frC$,
\begin{equation*}
\frC:=
\left( 
\begin{array}{cccc}
0&0&0&0\\
0&C^{11}&(C^{12})^T&0\\
0&C^{12}&C^{22}&0\\
0&0&0&0
\end{array}
\right).
\end{equation*}

The matrices $\frB$, $\frC$ are sparse as well as the finite element block $\cA$ of the matrix
$\frA$. The boundary element block $\cS$ of the matrix $\frA$ is a dense matrix.

\section{Numerical experiments}

We present a series of numerical examples for the $hp$-FE/BE coupling with     
Nitsche's method on uniform meshes, non-matching across the coupling interface and
meshes with geometric refinement towards the singularity. The numerical
experiments are designed to illustrate the $hp$-convergence behaviour predicted by the
analysis, in particular the robustness with respect to non-matching meshes and local polynomial
refinement.
First, we consider an example with a smooth solution and investigate convergence
of the $h$-version for different polynomial degrees. We also show experimentally
that the $p$-version converges with an exponential rate. Then, an example with a
singular solution will be presented for two configurations. 
We show a figure (cf. \ref{figSingC}) for the convergence of $h$- and $p$-versions for one
configuration and furthermore show that the $hp$-version with geometrically
refined meshes converges exponentially in Figures \ref{figSingA} and
\ref{figSingB}. For the convergence graphs we denote the total amount of degrees of freedom by
$N := N_{\textrm{FE}} + N_{\textrm{BE}}$, where $N_{\textrm{FE}}$ and $N_{\textrm{BE}}$ are the
number of degrees of freedom in the finite element domain and the boundary element domain, 
respectively.

\subsection{Example 1: smooth solution} \label{subsec:smooth}

In the first example we consider a square domain $\Omega := [-1,1]\times[-1,1] \subset \R^2$.
We introduce the boundary element domain $\Omega^2$ and its complement $\Omega^1$, where finite
elements are used (Fig. \ref{fig:hp_smth})
\begin{equation} \label{Omegadec}
\begin{array}{l}
\Omega^2:= [-1,0]\times[-\frac{1}{2},\frac{1}{2}], \\[1ex]
\Omega^1:= \Omega \setminus \Omega^2.
\end{array}
\end{equation}
\noindent The interface boundary is given by
\begin{equation} \label{Gamma_I}
\begin{array}{l}
\Gamma_I := \partial \Omega^1 \cap \partial \Omega^2.
\end{array}
\end{equation}
Let $G(\bx,\by) $ be the fundamental solution of the two-dimensional Laplace operator
\[
G(\bx,\by) := -\dfrac{1}{2\pi} \log|\bx-\by|.
\]
In particular for $\by \notin \overline \Omega$ there holds
\[
\Delta_\bx G(\bx,\by) = 0, \qquad \forall \bx \in \Omega \subset \R^2.
\]
We fix $\by=(-1,-2)$ and let $\bx=(x_1,x_2)$ be variable. We define
\begin{equation} \label{u1def}
u(\bx):=-2\pi\dfrac{\partial}{\partial x_1} G((x_1,x_2),(-1,-2)) = 
\dfrac{x_1+1}{(x_1+1)^2+(x_2+2)^2}.
\end{equation}
The function $u(\bx)$ is an exact solution of problem (\ref{p1}) with
\[
\Gamma_D:=\{-1\}\times[-1,1], 
\qquad \Gamma_N=\partial \Omega \setminus \Gamma_D,
\qquad f=0,
\qquad g = \dfrac{\partial u}{\partial n}\bigg|_{\Gamma_N}.
\]
Moreover, the function $u(\bx)$ is an exact solution of the interface problem (\ref{p2}) with
the decomposition (\ref{Omegadec}) and with the interface boundary (\ref{Gamma_I}).

In order to study the convergence of the method, we choose $\eta_0=2.0$ and perform a series of
experiments for the uniform $h$- and $p$-version. 

The numerical experiments for the $p$-version are obtained for fixed meshes with the mesh size
relation $h_1/h_2 = 4/5$ on the interface boundary with $20$ boundary elements and $32$ finite
elements with increasing $p_1\equiv p_2$. Since the exact solution $u(\bx)$ is analytic in 
$\Omega$, exponential convergence of the $p$-version is observed. The results, obtained for the
$h$-version with $p=1,2,3$, are compared with the $p$-version and are given in Fig.
\ref{fig:hp_smth}. 
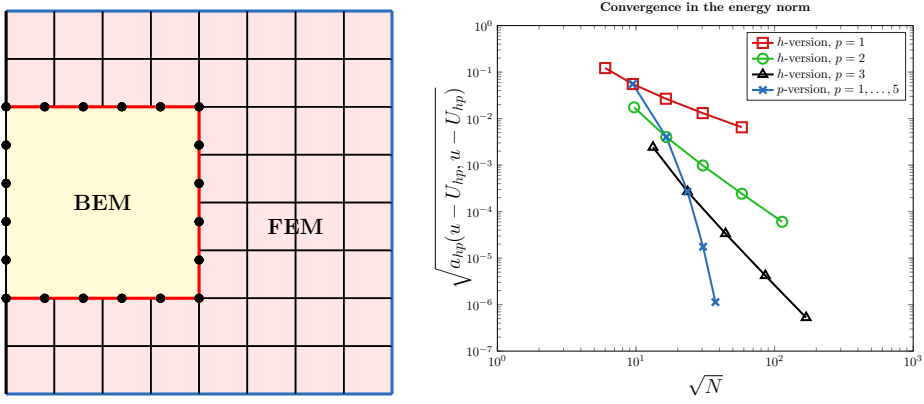
\begin{figure}[ht!] 
\begin{tabular}{cc}
\resizebox{0.4\textwidth}{!}{\includegraphics{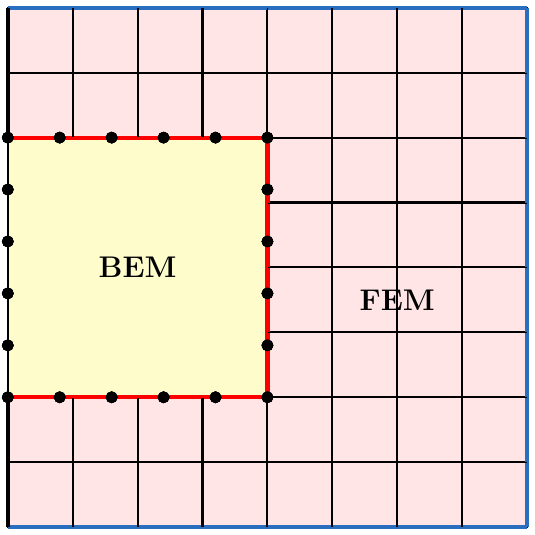}} &
\resizebox{0.5\textwidth}{!}{\begin{tikzpicture}
\definecolor{alexeyblue}{RGB}{40,110,191}
\begin{loglogaxis}[%
width=4.521in,
height=3.566in,
at={(0.758in,0.481in)},
scale only axis,
xmode=log,
xmin=1,
xmax=1000,
xminorticks=true,
xlabel={$\sqrt{N}$},
ymode=log,
ymin=1e-07,
ymax=1,
yminorticks=true,
ylabel={$\displaystyle\sqrt{a_{hp}(u-U_{hp},u-U_{hp})}$},
axis background/.style={fill=white},
title style={font=\bfseries},
title={Convergence in the energy norm},
label style={font=\Large},
legend style={at={(0.6,0.77)},anchor=south west,legend cell align=left,align=left,draw=white!15!black}
]
\addplot [color={rgb:red,100;green,10;blue,10},solid,ultra thick,mark=square,mark options={solid},mark size=4.0]
  table[row sep=crcr]{%
6	0.12217525043349\\
9.48683298050514	0.0554442739088191\\
16.431676725155	0.0267531753690621\\
30.2985148150862	0.0131478190392329\\
58.0172388174411	0.00651643832215361\\
};
\addlegendentry{$h$-version, $p=1$};

\addplot [color={rgb:red,5;green,30;blue,5},solid,ultra thick,mark=o,mark options={solid},mark size=4.0]
  table[row sep=crcr]{%
9.69535971483266	0.0174525010392136\\
16.5529453572468	0.00401330477823672\\
30.364452901378	0.000980195662143903\\
58.0517010948	0.000242671697681587\\
113.463650567043	6.03723000413043e-05\\
};
\addlegendentry{$h$-version, $p=2$};

\addplot [color=black,solid,ultra thick,mark=triangle,mark options={solid},mark size=4.0]
  table[row sep=crcr]{%
13.2664991614216	0.00242123320362288\\
23.5372045918796	0.000269554021887\\
44.2492937796752	3.3707120574865e-05\\
85.7787852560294	4.21021451366194e-06\\
168.896417960832	5.25555686092293e-07\\
};
\addlegendentry{$h$-version, $p=3$};

\addplot [color=alexeyblue,solid,ultra thick,mark=x,mark options={solid,line width=2pt},mark size=4.0]
  table[row sep=crcr]{%
9.48683298050514	0.0554442739088191\\
16.5529453572468	0.00401330477823672\\
23.5372045918796	0.000269554021887\\
30.4959013639538	1.75742048905652e-05\\
37.4432904536981	1.13052831288366e-06\\
};
\addlegendentry{$p$-version, $p=1,\dots,5$};

\end{loglogaxis}
\end{tikzpicture}%
} 
\end{tabular}
\caption{Convergence of the $h$-and $p$-versions for example 1.}\label{fig:hp_smth}
\end{figure}

\subsection{Example 2: singular solution} \label{subsec:sing}
For our second example we choose $\Omega$ to be an L-shaped domain
\begin{equation} \label{2:Omega}
\Omega:= [-1,1]^2 \setminus [0,1]^2
\end{equation}
and for the first configuration consider the decomposition 
\begin{equation} \label{2:Omegadec}
\begin{array}{l}
\Omega^2:= \{[-\frac{1}{2},\frac{1}{2}]\times[-\frac{1}{2},\frac{1}{2}] \}
\setminus \{[0,\frac{1}{2}]\times[0,\frac{1}{2}]\}, \\[1ex]
\Omega^1:= \Omega \setminus \Omega^2,
\end{array}
\end{equation}
as shown in Fig. \ref{figSingA}. 
Furthermore let
\begin{equation} \label{2:Gamma_I}
\begin{array}{c}
\Gamma_D:=\{\{0\} \times [0,1]\} \cup \{[0,1] \times \{0\}\}, \\[1ex]
\Gamma_N := \partial \Omega \setminus \Gamma_D,
\qquad \Gamma_I := \partial \Omega^1 \cap \partial \Omega^2.
\end{array}
\end{equation}
Note that in this decomposition the singularity at the origin is completely encapsulated in the
boundary element domain $\Omega^2$.

For a second configuration we consider $\Omega$ as above, but this time split the L-shaped
domain symmetrically at the line from $(-1,-1)^\top$ to $(0,0)^\top$ as depicted in Fig.
\ref{figSingB}, i.e.
\begin{equation}
\begin{array}{l}
\Omega^2 := \Omega\cap \{x,y\in[-1,1]: x-y\leq 0 \},\\[1ex]
\Omega^1 := \Omega\setminus \Omega^2,
\end{array}
\end{equation}
with $\Gamma_D,\Gamma_N$ and $\Gamma_I$ as in \eqref{2:Gamma_I}. In this case the singularity
is contained in both subdomains and is also at the end of interface $\Gamma_I$.
For this kind of domain $\Omega$, $r^{2/3}$ is a typical singularity situated at
the origin. 
As exact solution we choose 
\begin{equation*}
u(r,\vartheta) := r^{2/3} \sin(\tfrac{2}{3}(\vartheta-\tfrac{\pi}{2})),
\end{equation*}
where $(r,\vartheta)$ denote the polar coordinates of the plane.

It is easy to check, that $u(r,\vartheta)$ is an exact solution of (\ref{p1}) and (\ref{p2})
for (\ref{2:Omega}) with (\ref{2:Gamma_I}) in both configurations.
There holds 
\[
u \in H^1(\Omega), \quad u \notin H^2(\Omega).
\]
It is possible to show that $u\in H^{5/3-\varepsilon}(\Omega)$, $\forall \varepsilon >0$,
therefore the convergence rate $2/3-\epsilon$ for the $h$-version in the $H^1$-seminorm is
optimal. The theoretical convergence rate agrees with the numerical convergence rate 
$\approx 0.66$, as shown in Figure \ref{figSingC}.\\

It is known from the work of Stephan and Suri \cite{StSu91}, that the
convergence rate for the $p$-version of the BEM in the energy norm is twice that
the corresponding convergence rate of the $h$-version. Therefore, we expect the
convergence rate $4/3$ in the energy norm in our example (cf. Figure \ref{figSingC}).
Due to Theorem \ref{thm:apriori_sharp}, the $p$-version of our FE/BE Nitsche's coupling is
suboptimal, caused by the factor $p^{1/2}$. This is a disadvantage observed
theoretically in Theorem \ref{thm:apriori_sharp}, but could not be verfied in
the numerical experiments. 

\begin{figure}
\centering
\begin{tabular}{cc}
\resizebox{0.4\textwidth}{!}{\includegraphics{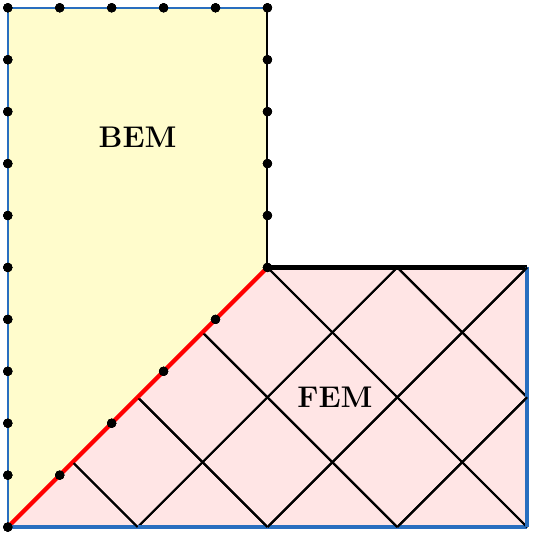}} &
\resizebox{0.52\textwidth}{!}{\includegraphics{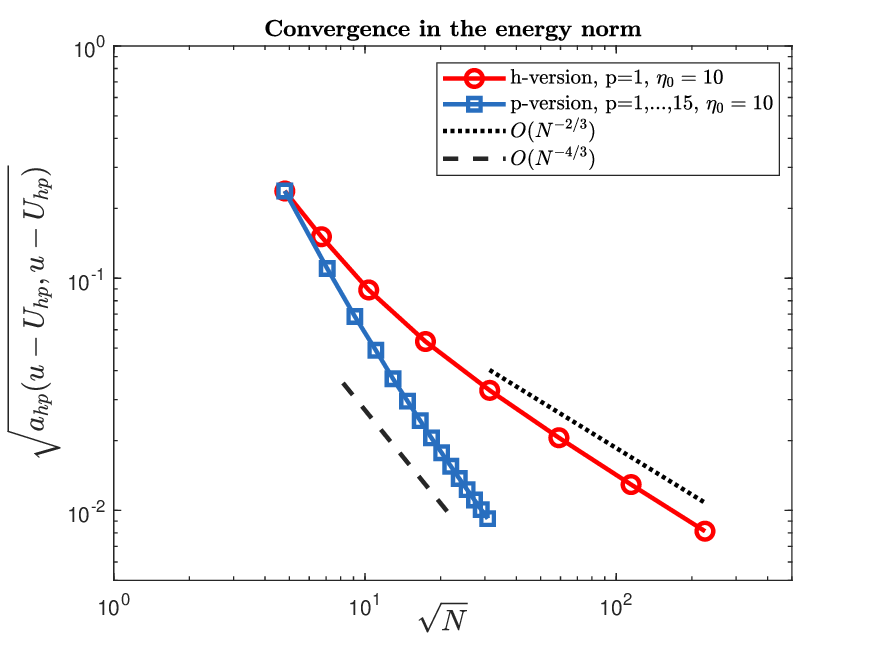}}
\end{tabular}
\caption{Convergence of the $h$-, $p$-versions for the split L-shape configuration.}
\label{figSingC}
\end{figure}

\begin{figure}
\centering
\begin{tabular}{cc}
\resizebox{0.4\textwidth}{!}{\includegraphics{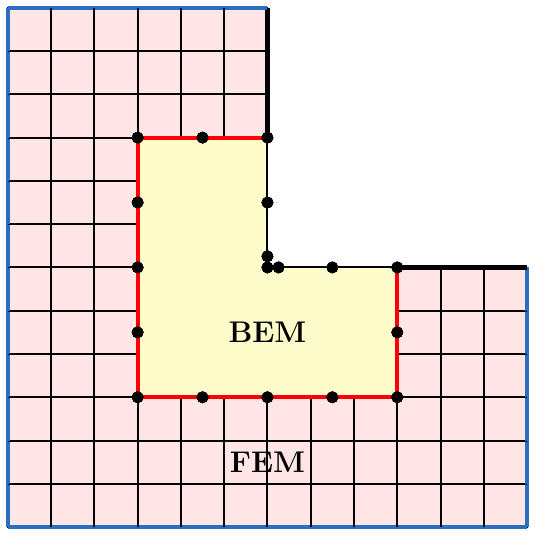}} &
\resizebox{0.56\textwidth}{!}{\includegraphics{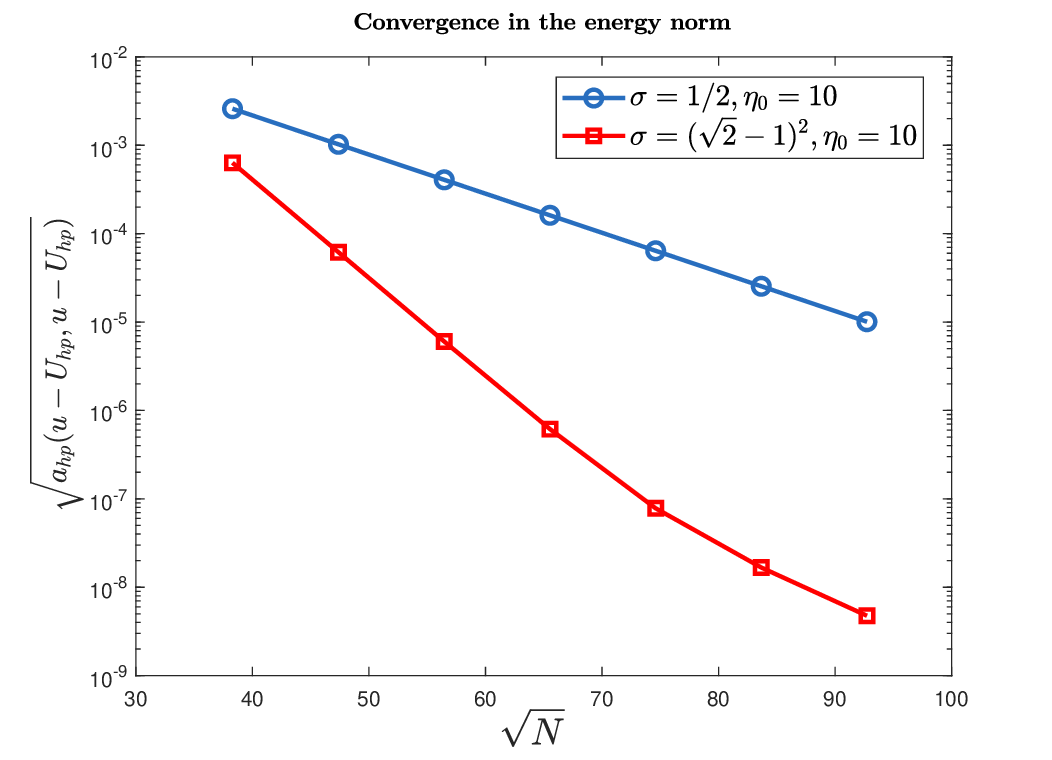}}	
\end{tabular}
\caption{Convergence of the $hp$-version for the first configuration.}\label{figSingA}
\end{figure}

\begin{figure}
\centering
\begin{tabular}{cc}
\resizebox{0.4\textwidth}{!}{\includegraphics{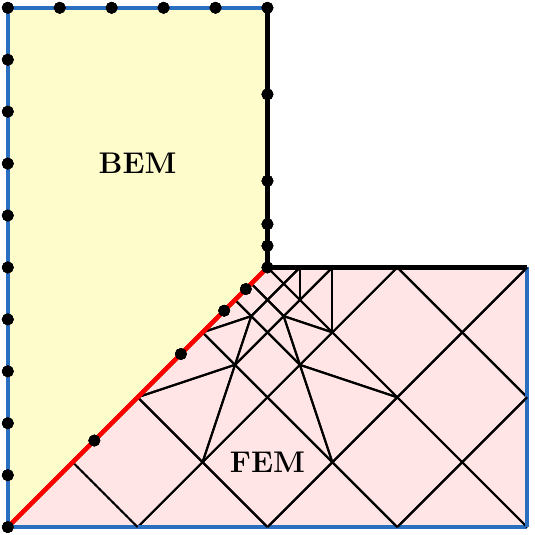}} &
\resizebox{0.5\textwidth}{!}{\includegraphics{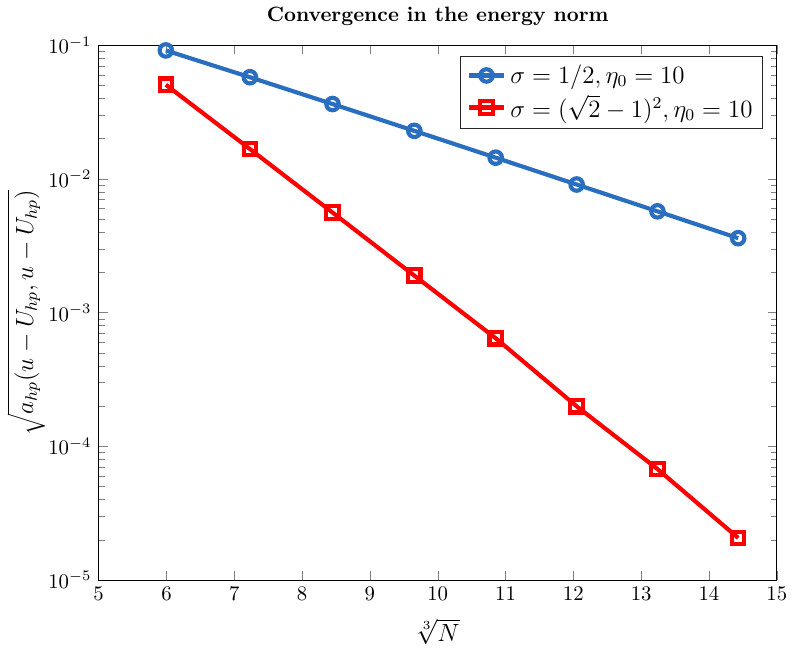}}	
\end{tabular}
\caption{Convergence of the $hp$-version for the split L-shape configuration.}\label{figSingB}
\end{figure}

For the given $hp$-configurations of the L-shaped domain in Figures \ref{figSingA} and
\ref{figSingB}, we observe exponential convergence as expected.
It is noteworthy, that encapsulating the singularity with the boundary element
subdomain as in the first configuration is advantageous from a computational standpoint, since
exponential convergence occurs with respect to the square root of the total degrees of freedom
$N$.
In the split L-shape configuration the exponential convergence scales with $\sqrt[3]{N}$ as the
finite element domain dominates the growth of the degrees of freedom, since $N_{\textrm{FE}}$
is proportional to $p^3$ in contrast to $N_{\textrm{BE}}\sim p^2$.
Due to the dependence of $\invrate$ on the local mesh size the conditioning of the
system matrix increases dramatically in the case of geometric refinement
of the coupling interface $\Gamma_I$.

\section*{Acknowledgements}
Peter Hansbo was supported by the Swedish Research Council Grant No. 2022-03908.

\bibliography{refs}

\end{document}